\documentclass{amsart}
\usepackage{amsthm}
\usepackage{amsmath}
\usepackage{amssymb}
\usepackage{enumerate}
\usepackage{hyperref}

\usepackage[foot]{amsaddr}

\usepackage{mathrsfs}

\theoremstyle{theorem}
\newtheorem{thm}{Theorem}
\newtheorem{cor}[thm]{Corollary}
\newtheorem{lem}[thm]{Lemma}
\newtheorem{prop}[thm]{Proposition}

\theoremstyle{definition}

\newtheorem{defn}[thm]{Definition}
\newtheorem{rmk}[thm]{Remark}

\newtheorem{cla}[thm]{Claim}

\newcommand{\N}{\mathbb{N}}

\newcommand{\Q}{\mathbb{Q}}
\newcommand{\R}{\mathbb{R}}

\newcommand{\Fc}{\mathcal{F}}

\newcommand{\Rc}{\mathcal{R}}
\newcommand{\Gc}{\mathcal{G}}

\renewcommand{\d}{\delta}
\newcommand{\e}{\epsilon}

\newcommand{\Id}{\operatorname{Id}}

\newcommand{\Bc}{\mathcal{B}}

\renewcommand{\Pr}{\text{Pr}}

\renewcommand{\tilde}{\widetilde}

\newcommand{\ds}{/ \!  /}

\newcommand{\id}{\operatorname{Id}}

\newcommand{\la}{\lambda}

\newcommand{\Ec}{\mathcal E}

\newcommand{\Gap}{\operatorname{Gap}}

\newcommand{\knum}{\mathsf x}

\begin{document}

\title[Dominated splittings in Banach spaces]{Characterization of dominated splittings for operator cocycles acting on Banach spaces}

\author{Alex Blumenthal$^1$}
\address{$^1$Courant Institute of Mathematical Sciences, New York University, 251 Street, New York, New York 10012, U.S.A.
}
\email{$^1$alex@cims.nyu.edu}
\author{Ian D. Morris$^2$}
\address{$^2$Mathematics department, University of Surrey, Guildford GU2 7XH, United Kingdom}
\email{$^2$i.morris@surrey.ac.uk}

\maketitle

\begin{abstract}
Versions of the Oseledets multiplicative ergodic theorem for cocycles acting on infinite-dimensional Banach spaces have been investigated since the pioneering work of Ruelle in 1982 and are a topic of continuing research interest. For a cocycle to induce a \emph{continuous} splitting in which the growth in one subbundle exponentially dominates the growth in another requires additional assumptions; a necessary and sufficient condition for the existence of such a dominated splitting was recently given by J. Bochi and N. Gourmelon for invertible finite-dimensional cocycles in discrete time. We extend this result to cocycles of injective bounded linear maps acting on Banach spaces (in both discrete and continuous time) using an essentially geometric approach based on a notion of approximate singular value decomposition in Banach spaces. Our method is constructive, and in the finite-dimensional case yields explicit growth estimates on the dominated splitting which may be of independent interest.
\end{abstract}


\section{Introduction}

The celebrated multiplicative ergodic theorem of Oseledets \cite{Os68} states that a cocycle of finite-dimensional linear maps acting on a linear bundle over a dynamical system -- for example, the derivative cocycle of a diffeomorphism, acting on the tangent bundle -- induces a measurable decomposition of the bundle into different subbundles each having a particular exponential growth rate. The further development  of this result --  weakening its hypotheses, shortening its proof, and extending its scope -- has continued in the literature in every decade since its publication (we note for example \cite{FrLlQu10,Ka87,Ka02,LiLu10,Ma83,Ra79,Ru82,Th97,Wa93}). In this article we are concerned with two modes of extension of Oseledets' result: firstly, its extension to infinite-dimensional cocycles; and secondly, the existence of a continuous, rather than just measurable, decomposition into subspaces with particular exponential growth rates.

The extension of Oseledets' theorem to cocycles of operators on infinite-dimensional spaces has been a problem of ongoing interest since the pioneering work of Ruelle in 1982 (see e.g. \cite{GoQu14a,LiLu10,Ma83,Ru82,Th92}), and has enjoyed a burst of recent attention motivated by applications including the study of transfer operators of random dynamical systems, stochastic partial  differential equations, and functional analysis (\cite{FrGoQu14,FrLlQu13,Mo12, flandoli1990stochastic}). By comparison the construction of continuous splittings has been a recent activity, even in low dimensions: the two-dimensional case was first published by Yoccoz in 2004 (\cite{Yo04}, see also \cite{DaFiLuYe14}) and the case of invertible real cocycles in finite dimension was subsequently investigated by J. Bochi and N. Gourmelon \cite{BoGo09}. In this article we establish some necessary and sufficient conditions for the existence of a dominated splitting -- that is, a continuous splitting in which the growth in one subbundle exponentially dominates the growth in another -- for cocycles of bounded injective operators acting on a Banach space in discrete or continuous time.

Let us briefly recall the content of Oseledets' theorem. Suppose that $T \colon X \to X$ is an ergodic invertible measurable transformation acting on a probability space $(X,\mathcal{F},\mu)$, and $V$ a real or complex vector space having finite dimension $d$, say. A function $A \colon X \times \mathbb{N} \to GL(V)$, which we will denote by $(x,n) \mapsto A_x^n$,  is called a \emph{cocycle} over $T$ if it has the property that $A^{n+m}_x=A^n_{T^mx}A^m_x$ for every $x \in X$ and $n,m \geq 0$. Subject to suitable integrability assumptions, Oseledets' theorem asserts that there exist real numbers $\lambda_1,\ldots,\lambda_k$ and measurable functions $\Ec_1,\ldots,\Ec_k$ from $X$ to the Grassmannian manifold of $V$ such that $V=\Ec_1(x)\oplus \cdots \oplus \Ec_k(x)$, and $\lim_{n \to \infty} \frac{1}{n}\log|A^n_xv|=\lambda_k$ for every nonzero vector $v \in \Ec_k(x)$, for $\mu$-almost every $x \in X$. The splitting $V=\Ec_1\oplus\cdots \oplus \Ec_k$ has the further property of being \emph{equivariant}: we have $A^n_x\Ec_k(x)=\Ec_k(T^nx)$ for almost every $x \in X$ and for every $n \geq 0$.

If $X$ is additionally a topological space (typically, a compact metric space) and $T$ a homeomorphism, we might further ask if the equivariant splitting $V=\Ec_1\oplus \cdots \oplus \Ec_k$ can be chosen to be continuous. However, if $T$ admits more than one invariant probability measure then the Lyapunov exponents $\lambda_i$ will in general be different with respect to different probability measures, and as a consequence the dimensions of the subspaces $\Ec_i(x)$, though constant almost everywhere for each ergodic measure, can fail to be the same for every $x \in X$. It is furthermore not difficult to construct examples in which the Lyapunov exponents $\lim_{n \to \infty}\frac{1}{n}\log |A^n_xv|$ fail to exist for certain $x \in X, v \in V \setminus \{0\}$ (even if $V$ is one-dimensional), or examples in which the Lyapunov exponents are all distinct and the Oseledets subspaces have constant dimension,  but in which continuous splittings cannot exist because the convergence of $\frac{1}{n}\log|A^n_x|$ is not uniform (see e.g. \cite{Fu97}). In view of these considerations the continuous analogue of Oseledets splittings which we investigate in this article is the \emph{dominated splitting}. A dominated splitting describes a situation in which $V$ admits a continuous equivariant splitting $V=\Ec(x) \oplus \Fc(x)$ in such a way that while the exponential growth rates of vectors in $\Ec(x)$ and $\Fc(x)$ are not necessarily constant with respect to $x$ (or even well-defined at all $x \in X$) each vector in $\Ec(x)$ nonetheless grows exponentially \emph{faster} under the action of $A^n_x$ than every vector in $\Fc(x)$.

Let us define this concept formally. Let $X$ be a topological space, $T \colon X \to X$ a homeomorphism, $(V,|\cdot|)$ a Banach space, and $L(V)$ the space of bounded linear operators on $V$. We let $\Gc(V)$ denote the Grassmannian of closed complemented subspaces of $V$ equipped with a certain natural topology which will be defined in the following section. A function $A \colon X \times \mathbb{N} \to L(V)$, denoted $(x,n)\mapsto A_x^n$, will be called a \emph{linear cocycle} if $A_x^{n+m}=A_{T^mx}^nA_x^m$ for every $x \in X$ and $n,m \geq 0$, and if additionally $A_x^n \equiv \mathrm{Id}_V$. 
We say that $A$ admits a \emph{dominated splitting} if there exist continuous functions $\Ec,\Fc \colon X \to \Gc(V)$ and constants $C>0$, $\tau \in (0,1)$ such that $V=\Ec(x)\oplus \Fc(x)$ for every $x \in X$, $A^n_x\Ec(x)=\Ec(T^nx)$ and $A^n_x\Fc(x)\subseteq \Fc(T^nx)$ for every $x \in X$ and $n \geq 1$, and additionally
\[\sup_{\substack{u \in \Ec(x), v \in \Fc(x)\\|u|=|v|=1}}\frac{|A^n_x v|}{|A^n_x u|}\leq C\tau^n \]
for every $x \in X$ and $n \geq 1$.  For finite-dimensional $V$, the notion of a dominated splitting (particularly for derivative cocycles of diffeomorphisms on manifolds) originates in its applications to smooth ergodic theory (\cite{Li80,Ma82,Pl72}) and continues to have profound applications (see e.g. \cite{BoVi05,GaWe06,PuSa09}). However, the general problem of constructing dominated splittings appears to have received relatively little attention. Recall that the singular values of a $d \times d$ real or complex matrix $A$ are the non-negative square roots of the eigenvalues of the positive semidefinite matrix $A^*A$, and are denoted $\sigma_1(A),\ldots,\sigma_d(A)$ in decreasing order, with repetition in the case of multiplicity. The following characterization of finite-dimensional dominated splittings was recently proved by J. Bochi and N. Gourmelon \cite{BoGo09}:
\begin{thm}[Bochi-Gourmelon]\label{thm:bogo}
Let $T \colon X \to X$ be a homeomorphism of a compact Hausdorff space and $A \colon X \times \mathbb{N} \to GL_d(\R)$ a continuous linear cocycle. Then the following are equivalent:
\begin{itemize}
\item
There exist $K>0$, $\tau \in (0,1)$ such that 
\begin{equation}\label{eq:bogo}\sigma_{k+1}(A^n_x)\leq K\tau^n \sigma_k(A^n_x)\end{equation}
for every $x \in X$ and $n \geq 0$;
\item
There exists a dominated splitting $\R^d=\Ec(x) \oplus \Fc(x)$ such that $\dim \Ec(x) = k$ for all $x \in X$.
\end{itemize}
\end{thm}
Results of this type have since found powerful applications to problems in matrix analysis (\cite{BoMo15,Mo10}) in which there is no clear a priori reason why the matrices $A^n_x$ should need to be invertible. It is thus of interest to ask whether Theorem \ref{thm:bogo} might be generalized to the case of cocycles of non-invertible matrices. Moreover, the infinite-dimensional extensions of Oseledets' theorem described earlier include numerous examples in which invertibility of the matrices or operators $A^n_x$ is not required. It is therefore both natural and desirable to attempt the extension of Theorem \ref{thm:bogo} to infinite-dimensional settings and to the context of non-invertible linear operators.

The concept of singular value defined above may be extended to infinite-dimensional Hilbert spaces in a relatively direct fashion, but in the context of general Banach spaces there is no precise analogue of the notion of singular value: instead a family of different generalizations exists, with such generalizations being referred to generically as \emph{$s$-numbers}. In this article we principally use one specific generalization, namely \emph{Gelfand numbers}: if $A$ is a bounded linear operator acting on a Banach space $(V,|\cdot|)$, then for each integer $k \geq 1$ we define the $k^{\mathrm{th}}$ Gelfand number of $A$ to be the quantity
\[c_k(A):= \inf\{|A|_F|\colon F\text{ is a closed }(k-1)\text{-dimensional subspace of }V\}. \] 
When $(V,|\cdot|)$ is a Hilbert space the Gelfand numbers of $A$ coincide precisely with its singular values. Our first main result generalizes Theorem \ref{thm:bogo} to injective, but not necessarily invertible, linear cocycles acting on Banach spaces:
\begin{thm}\label{thm:main}
Let $T$ be a homeomorphism of a compact topological space $X$ and let $A \colon X \times \mathbb{N} \to L(V)$ denote a cocycle of injective bounded linear operators acting on a Banach space $(V,|\cdot|)$. Then the following statements are equivalent:
\begin{itemize}
\item
There exist constants $K>0$ and $\tau \in (0,1)$ such that
\begin{equation}\label{eq:magic} \max\{c_{k+1}(A_x^n),c_{k+1}(A_{Tx}^n))\} < K\tau^n c_{k}(A_x^{n + 1})\end{equation}
for all $x \in X$ and $n \geq 1$;
\item
There exist a continuous equivariant splitting $V= \Ec (x) \oplus \Fc(x)$, for which $\dim \Ec(x) = k < \infty$ for all $x \in X$, and constants $K>0$, $\tau \in (0,1)$ such that 
\[\sup_{\substack{u \in \Ec(x), \,v \in \Fc(x)\\ |u|=|v|=1}} \frac{|A_x^n v |}{| A_x^n u| } \leq K\tau^n\]
for every $x \in X$ and $n \geq 1$.
\end{itemize}
\end{thm}

It was observed in \cite{BoMo15} that the condition \eqref{eq:bogo} is not sufficient to obtain the existence of a dominated splitting for non-invertible cocycles; instead we require the condition \eqref{eq:magic} to guarantee the existence of continuous upper and lower subbundles. The problem of establishing that $A^n_x$ is injective on the upper subbundle $\Ec$  seems to be more challenging, and it is this consideration which necessitates our hypothesis that $A^n_x$ be injective on $V$. We also draw the reader's attention to the fact that if every $A^n_x$ is invertible, then \eqref{eq:magic} is equivalent to the simpler condition $c_{k+1}(A^n_x)\leq K\tau^nc_k(A^n_x)$, which in the Hilbert space case is precisely Bochi and Gourmelon's condition \eqref{eq:bogo}.

We note as well that the proof of Theorem \ref{thm:main} carries over to the case when the cocycle $A$ acts on a continuous Banach bundle $\pi : E \to X$ (see \cite{lang1972differentiable} for a precise definition); in this terminology Theorem \ref{thm:main} is formulated for the trivial Banach bundle $E = X \times V$. The extension to general Banach bundles is a straightforward consequence of the methods given in this paper: all the arguments in the proof of Theorem \ref{thm:main} are `local' in $x$, and thus carry over to the bundle case by the local trivialization property. Details are left to the reader.

\medskip

Once Theorem \ref{thm:main} is established, it is a simple matter to derive a corresponding result for flows. Suppose that $(\phi^t)_{t \in \mathbb{R}}$ is a continuous flow on a topological space $X$; that is, suppose that $\{\phi^t\colon t \in \mathbb{R}\}$ is a set of homeomorphisms of $X$ such that $\phi^{t+s}=\phi^t\circ\phi^s$ for every $s,t \in \R$, and such that $t \mapsto \phi^tx$ is continuous for every $x \in X$. We shall say that a \emph{cocycle} over $(\phi^t)$ is a function $B \colon X \times [0,+\infty) \to L(V)$ such that $B^{s+t}_x=B_{\phi^tx}^sB_x^t$ for every $x \in X$ and $s,t \geq 0$. We obtain the following continuous-time analogue of Theorem \ref{thm:main}:
\begin{thm}\label{thm:continuousTime}
Let $(\phi^t)_{t \in \mathbb{R}}$ be a continuous flow on a compact topological space $X$ and let $B : X \times [0,+\infty) \to L(V)$ be a cocycle over $(\phi^t)$ such that $B_x^t$ is injective for every $x \in X$ and $t \geq 0$. We assume that
\begin{align}\label{eq:continuityHypotheses}\begin{split}
x & \mapsto B_x^t \qquad \text{ is continuous in the operator norm topology for each fixed $t \geq 0$, and}  \\
(x,t) & \mapsto B_x^t \qquad \text{ is continuous in the strong operator topology.}
\end{split}\end{align}
Then the following are equivalent:
\begin{itemize}
\item[(a)]
There exist $C,\gamma>0$ such that for any $x \in X, t \geq 0$,
\begin{align}\label{eq:continuousTime}
\sup_{0 \leq \e \leq 1} c_{k+1}(B_{\phi^{\e} x}^t) \leq Ce^{-\gamma t}  c_k(B_x^{t+1}) \, .
\end{align}
\item[(b)]
There exist a continuous equivariant splitting $V=\Ec(x)\oplus \Fc(x)$, for which $\dim \Ec(x) = k < \infty$ for all $x \in X$, and constants $C,\gamma>0$ such that for all $t\geq 0$,
\begin{align}\label{eq:continuousDS}
\sup_{\substack{u \in \Ec(x), v\in \Fc(x)\\ |u|=|v|=1}}\frac{|B_x^tv|}{|B_x^tu|}\leq Ce^{-\gamma t} \, .
\end{align}
\end{itemize} 
\end{thm}
We note that if $B_x^t$ is assumed invertible for every $x \in X$ and $t \geq 0$, then similarly to Theorem \ref{thm:main} the condition \eqref{eq:continuousTime} may be replaced with the simpler hypothesis $c_{k+1}(B_x^t)\leq C e^{- \gamma t} c_k(B_x^t)$. As is the case with Theorem \ref{thm:main}, an analogous statement holds when the cocycle $B$ acts on a continuous Banach bundle.

\medskip

The methods which we use to prove Theorem \ref{thm:main} also yield new information even in the invertible, finite-dimensional context of Theorem \ref{thm:bogo}. Specifically, the proof of Theorem \ref{thm:bogo} by Bochi and Gourmelon is nonconstructive, relying heavily on ergodic theory and compactness both to construct the equivariant splitting and to show that it is dominated; in particular, Bochi and Gourmelon's statement does not yield an explicit relationship between the constants $K, \tau$ in the definition of dominated splitting and the constants $K,\tau$ in the condition \eqref{eq:bogo}. The use of ergodic theory also necessitated the restriction of their results to compact Hausdorff spaces as opposed to more general topological spaces. By adapting the proof of Theorem \ref{thm:main} to the finite-dimensional context we are able to obtain a quantitative version of the Bochi-Gourmelon result which does not even require compactness of the base dynamical system.

Below, we have endowed $\R^d$ with the standard Euclidean inner product and corresponding norm $\|\cdot\|$.
\begin{thm}\label{thm:finiteDimensional}
Let $T$ be a homeomorphism of a topological space $X$, $d \in \N$, and let $A : X \times \N \to L(\R^d)$ be a continuous cocycle of invertible linear operators for which $\kappa := \sup_{x \in X} \|A_x\| \cdot \sup_{x \in X} \|A^{-1}_x\| < \infty$.

If $A$ satisfies the condition \eqref{eq:bogo} for constants $K > 0, \tau \in (0,1)$ and some $k < d$, then there exist a constant $\Rc_\Ec > 0$ and a continuous, $A$-equivariant splitting $\R^d = \Ec(x) \oplus \Fc(x)$ for $x \in X$ with the property that
\begin{gather*}
\|(A^n_x|_{\Ec(x)})^{-1}\| \leq R_{\Ec}^{-1} \, \big( \sigma_k(A^n_x)\big)^{-1} \, , \quad \|A^n_x|_{\Fc(x)}\| \leq R_{\Ec}^{-2} \, \sigma_{k + 1}(A^n_x) \, ,  \\
\quad \text{ and } \quad \| \pi_x \| \leq R_{\Ec}^{-1} \, ;
\end{gather*}
here $\pi_x : \R^d \to \Ec(x)$ refers to the projection operator onto $\Ec(x)$ parallel to $\Fc(x)$. The constant $R_{\Ec} \in (0,1)$ is given by
\[
R_{\Ec} = \inf_{\substack{x \in X \\ n \geq 1}} \frac{|\det(A^n_x|_{\Ec(x)})|}{\prod_{i = 1}^k \sigma_i (A^n_x)} \, 
\]
and satisfies
\[
\log R_{\Ec} \geq - \bigg(2 k \log \kappa \cdot \bigg \lceil \frac{\log (3 \kappa^3 K)  -  \log (1 - \tau) }{- \log \tau} \bigg \rceil + \frac{36 k}{ 1 - \tau} \bigg).
\]
\end{thm}

\medskip

By far the more difficult half of Theorem \ref{thm:main} is the construction of the equivariant splitting $V=\Ec\oplus \Fc$ given the inequality \eqref{eq:magic} between Gelfand numbers. The range of methods for constructing such splittings in Oseledets-type theorems is quite varied, especially for cocycles which act on Banach spaces. Our approach in this article is heavily influenced by the geometric ideas of M.S. Raghunathan \cite{Ra79} and D. Ruelle \cite{Ru79,Ru82} who, working in finite dimensions or in the Hilbert space context, constructed the components of measurable splittings as limits as $n \to \infty$ of subspaces associated to the singular value decomposition of $A^n_x$. Whilst subsequent authors such as R. Ma\~n\'e, Ph. Thieullen, Z. Lian and K. Lu have largely avoided this approach (\cite{Ma83,LiLu10,Th97}), it can be adapted to the Banach space context with surprising effectiveness (\cite{Bl15,GoQu14b}). The relevant material from the geometry of Banach spaces -- including a construction of an approximate singular value decomposition using Gelfand numbers, which may be of independent interest -- is presented in \S\ref{sec:geometry} below. In \S\ref{sec:hardbit} we prove Theorem \ref{thm:main}, and in \S\ref{sec:continuous} we deduce its continuous-time analogue, Theorem \ref{thm:continuousTime}.  In \S \ref{sec:finiteDimension}, we show how the methods of \S\ref{sec:hardbit} can be adjusted to produce the explicit estimates of Theorem \ref{thm:finiteDimensional}. 


\section{Preliminaries on Banach space geometry}\label{sec:geometry}

In this section, we give some necessary preliminaries on Banach space geometry.  Throughout, $V$ is a Banach space with norm $|\cdot|$. When $V'$ is another Banach space, we write $L(V, V')$ for the set of bounded linear operators, and we write $L(V) = L(V,V)$. 

When $E \subset V$ is a subspace, we write 
\[
m(A|_E) := \inf\{|A v| : v \in E, |v| = 1\} 
\]
for the minimal norm of $A|_E$, and when $w_1, \cdots, w_q \in V$, we write $\langle w_1, \cdots, w_q \rangle$ for the span of these vectors.

When $E, F$ are subspaces of $V$ for which $V = E \oplus F$, we write $\pi_{E \ds F}$ for the projection onto $E$ parallel to $F$. That is, $\pi_{E \ds F}|_E$ acts as the identity $\id_E$ on $E$, and $\ker \pi_{E \ds F} = F$.


%
%

\subsection{Basic preliminaries}

We denote by $\Gc(V)$ the Grassmanian of closed subspaces of $V$, endowed with the Hausdorff metric $d_H$, defined for $E, E' \in \Gc(V)$ by
\[
d_H(E, E')  = \max\{ \inf\{d(e, S_{E'}) : e \in S_E\} , \inf\{d(e', S_{E}) : e' \in S_{E'} \}\} \, ,
\]
where $S_E := \{v \in E : |v| = 1\}$.

We denote by $\Gc_q(V)$, $\Gc^q(V)$ the Grassmanians of $q$-dimensional and closed $q$-codimensional subspaces, respectively.
\begin{prop}[Chapter IV, \S 2.1 of \cite{Ka95}] \label{prop:grassProps}
The metric space $(\Gc(V), d_H)$ has the following properties.
\begin{enumerate}
\item $(\Gc(V), d_H)$ is complete. 
\item The subsets $\Gc_q(V), \Gc^q(V)$ are closed in $(\Gc(V), d_H)$.
\end{enumerate}
\end{prop}

It is actually somewhat inconvenient to compute directly with $d_H$, and so frequently it will be easier to use the \emph{gap}, defined for $E, E' \in \Gc(V)$ by
\[
\Gap (E, E') := \sup_{e \in S_E} d(e, E').
\]
Below, for a set $S \subset V$ we write $S^{\circ} = \{\ell \in V^* : \ell(v) = 0 \text{ for all } v \in S\}$ for the annihilator of $S$. For $a, b \in \R$ we write $a \vee b = \max\{a, b\}$.
\begin{lem} \label{lem:apertureProps}
Let $E, E' \in \Gc(V)$.
\begin{enumerate}
\item $\Gap(E, E') \vee \Gap(E', E) \leq d_H(E, E') \leq 2 (\Gap(E, E') \vee \Gap(E', E))$.
\item $\Gap(E, E') = \Gap(E'^{\circ}, E^{\circ})$.
\item Let $q \in \N$, and assume either that $E, E' \in \Gc_q(V)$ or that $E, E' \in \Gc^q(V)$. If $\Gap(E, E') < \frac{1}{q}$, then
\begin{align}\label{eq:gapEstimate}
\Gap(E', E) \leq \frac{q \Gap(E, E')}{1 - q \Gap(E, E')}.
\end{align}
\end{enumerate}
\end{lem}
Parts (1) and (2) can be found in Chapter IV, \S 2.1 and \S 2.3 of \cite{Ka95}, respectively. Part (3) is proved in \cite{Bl15}.

We say that $E, F \in \Gc(V)$ are \emph{complements} if $V = E \oplus F$; in this case, the projection operator $\pi_{E \ds F}$ associated to this splitting is automatically a bounded operator by the Closed Graph Theorem.

\begin{defn}\label{defn:angle}
Let $E, F \in \Gc(V)$. The \emph{minimal angle} $\theta(E, F) \in [0,\frac{\pi}{2}]$ from $E$ to $F$ is defined by
\[
\sin \theta (E, F) = \inf\{|e - f | : e \in S_E, f \in F\}.
\]
\end{defn}
Roughly speaking, the minimal angle $\theta(E, F)$ will be small whenever $E$ is inclined towards $F$. Note that $\theta(E, F)$ may not equal $\theta(F, E)$. However, when $E, F$ are complemented, we have the formula (see \cite{Bl15})
\[
\sin \theta(E, F) = |\pi_{E \ds F}|^{-1} \, ,
\]
and so in this case, since $|\pi_{E \ds F}| \leq 1 + |\pi_{F \ds E}| \leq 2 |\pi_{F \ds E}|$, we have $\sin \theta(F, E) \leq 2 \sin \theta(E, F)$.

\medskip

Complementation is an open condition in $\Gc(V)$.
\begin{lem}\label{lem:openCond}
Let $E, F \in \Gc(V)$ be complements. If $E' \in \Gc(V)$ is such that $d_H(E, E') < \sin \theta(E, F)$, then $E'$ and $F$ are complemented, and
\[
|\pi_{E' \ds F}| \leq \frac{|\pi_{E \ds F}|}{1 - |\pi_{E \ds F}| d_H(E, E')} \, .
\]
Consequently, being complemented is an open condition in $(\Gc(V), d_H)$.
\end{lem}
A proof of lemma \ref{lem:openCond} is given in \cite{BlYo15}.

It is not necessarily true that every $E \in \Gc(V)$ has a complement (unless $V$ is a Hilbert space). However, complements always exist for members of $\Gc_q(V), \Gc^q(V)$ for any $q \geq 1$:
\begin{lem}[III.B.10 and III.B.11 of \cite{Wo96}] \label{lem:compExist}
For any $E \in \Gc_q(V)$, there exists a complement $F$ for $E$ such that $\sin \theta(E, F) \geq \frac{1}{\sqrt{q}}$, i.e., $|\pi_{E \ds F}| \leq \sqrt{q}$.

For any $F \in \Gc^q(V)$, there exists a complement $E$ for $F$ such that $\sin \theta(E, F) \geq \frac{1}{\sqrt{q} + 1}$, i.e., $|\pi_{E \ds F}| \leq \sqrt{q} + 1$.
\end{lem}

\begin{rmk}\label{rmk:gapEst}
The following method will frequently be used for estimating the gap $\Gap(\cdot, \cdot)$. Let $E, F \in \Gc(V)$ be complements, and let $E' \in \Gc(V)$. Then, we have the simple estimate
\[
\Gap (E', E) \leq |\pi_{F \ds E} |_{E'}| \, .
\]
In particular, if $E$ and $E'$ are both $q$-dimensional (or both closed and $q$-codimensional), then since $d_H(E,E')\leq 2$ we have
\[d_H(E',E)\leq 4q|\pi_{F \ds E} |_{E'}| \]
by considering separately the cases in which $|\pi_{F \ds E} |_{E'}| <1/2q$ and otherwise.
This bound will be used time and again, in various guises, throughout the paper.
\end{rmk}

\subsection{The induced volume and determinants for operators on Banach spaces}

In this section we discuss the induced volume on finite-dimensional subspaces. Much of this material is given elsewhere \cite{Bl15, BlYo15}, and so no proofs are given in this section.

Let $E \subset V$ be a finite dimensional subspace. We define the \emph{induced volume} $m_E$ to be the Haar measure on $E$ normalized so that
\[
m_E \{v \in E : |v| \leq 1\} = \omega_{\dim E} \, ,
\]
where for $q \in \N$ we define $\omega_q$ to be the (standard) volume of the Euclidean unit ball in $\R^q$.

\begin{lem}\label{lem:indVol}
Let $q \in \N$ and let $E \subset V$ be a $q$-dimensional subspace. The induced volume $m_E$ satisfies the following.
\begin{enumerate}
\item For any $v \in E$ and any Borel $B \subset E$, we have $m_E(v + B) = m_E (B)$.
\item If $m'$ is any other non-zero, translation-invariant measure on $E$, then $m', m_E$ are equivalent measures. The Radon-Nikodym derivative $\frac{d m'}{dm_E}$ is constant on $E$, and equals $\frac{m'(B)}{m_E(B)}$, where $B \subset E$ is any Borel set with positive $m_E$-measure.
\item For any $a > 0$ and any Borel measurable set $B \subset E$, we have $m_E (a B) = a^q m_E(B)$.
\item Let $w_1, \cdots, w_q$ is any set of vectors in $E$, and write $P[v_1,\cdots,v_q] = \{\sum_{i = 1}^q a_i v_i : 0 \leq a_i \leq 1\}$ for the parallelepiped spanned by $\{v_1, \cdots, v_q\}$. Then, for any $\lambda_1, \cdots, \lambda_k \in \R$,
\[
m_E P[\la_1 w_1, \la_2 w_2, \cdots, \la_q w_q] = \bigg( \prod_{i = 1}^k |\la_i| \bigg)  m_E P[w_1, \cdots, w_q] \, .
\]
\end{enumerate}
\end{lem}

The determinant in this setting is defined in terms of ratios of induced volumes, as follows. Below, we have written $B_E = \{v \in E : |v| \leq 1\}$.
\begin{defn}\label{defn:det}
Let $A \in L(V, V')$ be a map of Banach spaces, and let $E \subset V$ be a finite dimensional subspace. We define the determinant $\det(A | E)$ of $A$ on $E$ by
\[
\det (A|E) = 
\begin{cases}
\frac{m_{A E}(A B_E)}{m_E(B_E)} & A|_E \text{ injective,} \\
0 & \text{otherwise.}
\end{cases}
\]
\end{defn}
\noindent We emphasize that this notion of determinant is \emph{unsigned}; for finite-dimensional inner product spaces this notion of determinant actually refers to the absolute value of the determinant.

For the remainder of this section we recall salient properties of the determinant. We begin by recalling how the determinant behaves under splittings:

\begin{lem}[Lemma 2.15 in \cite{Bl15}] \label{lem:detSplit}
Let $A \in L(V)$ and $E \subset V$ have dimension $q$. Let $E = G \oplus H$ be a splitting with $\dim G = l < q$, and assume that $A|_E$ is injective. Writing $E' = A E, G' = A G, H' = A H$, we have the estimate
\[
C^{-1} \big( \sin \theta(G', H') \big)^l  \leq \frac{\det(A | E)}{\det(A | G) \det(A | H)} \leq C \big( \sin \theta(G, H) \big)^{- l} \, ,
\]
where the constant $C \geq 1$ depends only on $q$.
\end{lem}
\noindent Lemma \ref{lem:detSplit} is standard for determinants arising from inner products; indeed, in that setting, the constant $C$ appearing above may be taken equal to $1$.

\medskip

Our next result concerns the relationship between the determinant and a suitable notion of `singular value'. Let $q \in \N$ and $A \in L(V)$. We define the maximal $q$-dimensional volume growth $V_q(A)$ by
\[
V_q(A) := \sup\{\det (A | E) : E \in \Gc_q(V)\} \, .
\]
When $V$ is a Hilbert space, $V_q(A)$ coincides with the product of the top $q$ singular values for $A$. To achieve an analogous result in our Banach space setting, we employ the following generalization of `singular value':

\begin{defn}
Let $q \in \N$. For $A \in L(V)$, we define the $q$-th Gelfand number $c_q(A)$ by
\[
c_q(A) = \inf\{|A|_F| : F \in \Gc^{q - 1}(V) \} \, .
\]
\end{defn}
\noindent When $V$ is a Hilbert space, $c_q(A)$ coincides with the $q$-th singular value. The Gelfand number is only one example of a possible extension of the concept of singular value to the Banach space setting; see Pietsch \cite{Pi87} for a thorough exposition of such extensions, called $s$-numbers in the literature. The bound $c_k(ABC)\leq |A|c_k(B)|C|$ for linear operators $A, B, C$ on Banach spaces is obvious and will be used without comment.

Less frequently, we will rely on another notion of singular value, that of Kolmogorov number; for $A \in L(V)$, we define
\[
\knum_q(A) := \sup\{m(A|_W) : W \in \Gc_q(V)\} \, .
\]

We now formulate the connection between $V_q$, the Gelfand numbers $c_q$, and the Kolmogorov numbers $\knum_q$.
\begin{lem}\label{lem:gelfandProps}
Let $q \in \N$. There exists a constant $C$, depending only on $q$, for which the following holds for all $A \in L(\Bc)$:
\begin{gather} 
\label{eq:alternativeSnumber} \frac1C \knum_q(A) \leq c_q(A) \leq C \knum_q(A) \, , \\
\label{eq:prodGelfand} \frac1C \, c_{q}(A) V_{q-1}(A) \leq V_q(A) \leq C \, c_{q}(A) V_{q-1}(A) \, .
\end{gather}

\end{lem}
\noindent Item 2 in Lemma \ref{lem:gelfandProps} is proved in \cite{Bl15}. Item 1 follows from arguments in \cite{GoQu14a} (in \cite{GoQu14a}, the notation $F_q$ is used for the $q$-th Kolmogorov number) and the fact (proved in \cite{Bl15}) that for a $q$-dimensional subspace $E$ spanned by vectors $v_1, \cdots, v_q$, we have
\[
m_E  P[v_1, \cdots, v_q]  \approx |v_q| \cdot \prod_{i = 1}^q d(v_i, \langle v_{i + 1}, \cdots, v_q \rangle ) \, ;
\]
here, $P[v_1, \cdots, v_q] = \{a_1 v_1 + \cdots + a_q v_q : a_i \in [0,1], 1 \leq i \leq q\}$ is the unit parallelogram spanned by $v_1, \cdots, v_q$, and $\approx$ denotes `equal up to a multiplicative constant'. 

\medskip

We finish this part with a formulation of the local Lipschitz property of $\det$. Although technical, Lemma \ref{lem:detReg} is crucially important.
\begin{lem}[Proposition 2.15 in \cite{BlYo15}]\label{lem:detReg}
For any $q \geq 1$ and any $M > 1$ there exist $L_2, \d_2 > 0$ with the following properties. Let $V, V'$ be Banach spaces,  $A_1, A_2 \in L(V, V')$, and $E_1, E_2 \subset V$ be $q$-dimensional subspaces. Assume that 
\begin{gather*}
|A_i|,~ |(A_i|_{E_i})^{-1}| \leq M \quad i=1,2 \, ,\\
|A_1 - A_2|, ~d_H(E_1, E_2) \leq \d_2 \, .
\end{gather*}
Then, we have the estimate
\begin{equation} \label{regularity}
\left| \log \frac{\det(A_1 | E_1)}{\det (A_2 | E_2)} \right| \leq L_2 (|A_1 - A_2| + d_H(E_1, E_2)) \, .
\end{equation}
\end{lem}

\subsection{A version of the SVD for Banach space operators}

We pass now to the following result, a version of the Singular Value Decomposition for maps of Banach spaces. 

\begin{prop}\label{prop:genSVD}
For any $k \in \N$, $r \in (0,1)$, there exists a constant $D = D(k, r) > 1$ with the following properties.

Let $V, V'$ be Banach spaces, $A \in L(V, V')$, and $E \in \Gc_k(V)$ for which $E \cap \ker A = \{0\}$ and 
\[
\det(A | E) \geq r V_k(A) \, .
\]

Then, there exists $k$-codimensional subspaces $F\subset V, F' \subset V$ complementing $E$ and $E' := A E$, respectively, such that
\begin{enumerate}
\item $A F \subset F'$,
\item $|\pi_{E \ds F}|, |\pi_{E' \ds F'}| \leq D$,
\item and $|A|_F| \leq D c_{k + 1}(A)$.
\end{enumerate}
\end{prop}

Before going on to the proof, we discuss the meaning of Proposition \ref{prop:genSVD}. As a special case, let $A \in L(V, V')$ be of rank $\geq k$ and let $E \in \Gc_k(V)$ be such that $\det(A | E)$ is approximately $V_k(A)$, i.e., $A|_E$ realizes nearly all possible $k$-dimensional volume growth of $A$. It is not hard to show (see \eqref{eq:ckRealizedbyE} below) that for such a subspace $E$ we have $m(A|_E) = \min\{|A v| : v \in E, |v| = 1\} \geq C_k^{-1} c_k(A)$, where $C_k > 1$ depends only on $k$. 

So, the subspace $E$ can be thought of as the Banach space analogue of the top $k$-dimensional singular value subspace of a compact operator on a Hilbert space. To complete the analogy, one would hope that there exist $k$-codimensional complements $F, F'$ to $E, E' := A E$, respectively, for which (a) $A F \subset F'$, (b) $|\pi_{E \ds F}|$ and $|\pi_{E' \ds F'}|$ are controlled\footnote{We note that controlling $|\pi_{E \ds F}|$ is somehow the most difficult part of this proof. Indeed, this control is responsible for the necessity of the tedious procedure of paring off each dimension one at a time.}, and (c) $|A|_F|$ is controlled by $c_{k + 1}(A)$. Proposition \ref{prop:genSVD} yields all this information.

\begin{cor}[Singular Value Splitting for Banach Space Operators]\label{cor:SVD}
For any $k \in \N$ there exists a constant $C_k > 1$ with the following property. Let $A \in L(V, V')$ have rank $\geq k$. Let $E \subset V$ be such that $det(A | E) \geq \frac12 V_k(A)$. Then, there exists a complement $F \in \Gc^k(V)$ for $E$ and a complement $F' \in \Gc^k(V')$ for $E' := A E$ such that
\begin{itemize}
\item $A F \subset F'$,
\item $m(A|_E) \geq C_k^{-1} c_k(A)$, $|A|_F| \leq C_k c_{k + 1}(A)$, and
\item $|\pi_{E \ds F}|, |\pi_{E' \ds F'}| \leq C_k$ .
\end{itemize}
\end{cor}
\begin{proof}
Fix $E$ of dimension $k$ for which $\det( A | E) \geq \frac12 V_k(A)$ and apply Proposition \ref{prop:genSVD}. That $m(A|_E) \geq C_k^{-1} c_k(A)$ is deduced in the proof of Proposition \ref{prop:genSVD}; see \eqref{eq:ckRealizedbyE}.
\end{proof}

\medskip

We note that standard formulations of the Singular Value Decomposition for Hilbert spaces involves not just a splitting, but also an orthonormal \emph{basis} $\{v_i\}_{i = 1}^k$ of the upper subspace realizing each of the top $k$ singular values $\sigma_1(A), \cdots, \sigma_k(A)$ (in that setting, each $v_i$ is a normalized eigenvector for $A^* A$ with eigenvalue $\sigma_i(A)^2$). Although this aspect is not a part of our formulations, we point out that a basis of vectors $\{v_i\}_{i =1}^k$ is constructed during the course of the proof of Proposition \ref{prop:genSVD} for which (i) $|A v_i| \approx c_i(A)$ for each $i \leq k$, and (ii) the vectors $v_i$ are not too inclined towards each other (in the sense of the minimal angle $\theta$ defined in Definition \ref{defn:angle}).

\medskip

More generally, Proposition \ref{prop:genSVD} allows one to find such an $F$ even when, perhaps, $\det(A|E)$ realizes only some small proportion of the $k$-dimensional volume growth of $A$. The dividend of Proposition \ref{prop:genSVD} in this case is that one can still realize a version of the singular value splitting, so long as one is willing to accept a controlled amount of `error' encapsulated by the constant $D$.


The proof of Proposition \ref{prop:genSVD} is iterative, and proceeds by paring off one dimension at a time from the subspace $E$. A single iteration of this procedure is formulated below.

\begin{lem}\label{lem:oneStep}
Let $V, V'$ be Banach spaces, $A \in L(V, V')$ a nonzero operator, and $v$ a unit vector $\notin \ker A$. Then, there exist closed complements $G, G'$ to $\langle v \rangle, \langle A v \rangle$, respectively, such that
\begin{enumerate}
\item $A G \subset G'$,
\item $|\pi_{\langle v \rangle \ds G}| \leq  c_1(A)/|A v|,$ and $|\pi_{\langle A v \rangle \ds G'}| = 1$,
\item and for any $l \in \N$, 
\begin{align}\label{eq:volGrowRelation}
C^{-1} \frac{V_{l + 1}(A)}{c_1(A)} \leq V_{l }(A|_G) \leq C \frac{V_{l + 1}(A)}{|A v|} \, ,
\end{align}
where $C$ is a constant depending on $l$ alone.
\end{enumerate}
\end{lem}
\begin{proof}[Proof of Lemma \ref{lem:oneStep}]
We begin by using Lemma \ref{lem:compExist} to choose a closed complement $G'$ to $A v$ for which $|\pi_{\langle A v \rangle \ds G'}| = 1$. Defining $G := \{w \in V : A w \in G'\}$, we claim that $G$ is a complement to $\langle v \rangle$ for which
\begin{align}\label{eq:projectionFormtoshowComplement}
\pi_{\langle v \rangle \ds G} = \big( A|_{\langle v \rangle } \big)^{-1} \circ \pi_{\langle A v \rangle \ds G'} \circ A \, .
\end{align}
Indeed, the operator on the RHS is well-defined on $V$, bounded, has image $\langle v \rangle$, and acts as the identity on $\langle v \rangle$, and so gives rise to a well-defined projection operator onto $\langle v \rangle$. It now remains to check that the operator defined in \eqref{eq:projectionFormtoshowComplement} has kernel $G$, which we leave to the reader. Note that by \eqref{eq:projectionFormtoshowComplement}, we have the estimate $|\pi_{\langle v \rangle \ds G}| \leq c_1(A)/|A v|$. 

It remains to check \eqref{eq:volGrowRelation} for each $l \in \N$. Observe that without loss, we may assume that the rank of $A$ is $\geq l + 1$: if not, then $A|_G$ must have rank $\leq l - 1$, and so \eqref{eq:volGrowRelation} will hold vacuously. For the remainder of the proof, $l \in \N$ is fixed, $A$ is assumed to have rank $\geq l + 1$, and the symbols $\lesssim, \gtrsim$ mean `$\leq, \geq$ (respectively) up to a multiplicative constant depending only on $l$'.

For the left-hand inequality, let $H \subset V$ be an $(l + 1)$-dimensional subspace for which $\det(A | H) \geq \frac12 V_{l + 1}(A)$. Observe that $H \cap G$ has dimension $\geq l$, and so fix any $l$-dimensional subspace $J \subset H \cap G$ and any $1$-dimensional complement $K \subset H$ to $J$ for which $|\pi_{K \ds J}| = 1$ as an operator on $H = J \oplus K$ (by Lemma \ref{lem:compExist}). Applying Lemma \ref{lem:detSplit} to $\det(A | H) = \det(A | K \oplus J)$, we estimate
\[
V_{l + 1}(A) \leq 2 \det(A | H) \lesssim |\pi_{K \ds J}| \det(A | K) \det(A | J) \leq c_1(A)  V_{l}(A|_G) \, . 
\]

For the right-hand inequality, let $L \subset G$ be an $l$-dimensional subspace for which $\det(A | L) \geq \frac12 V_l(A|_G)$. In particular, $A L$ has dimension $l$, and $|\pi_{\langle A v \rangle \ds A L}| = |\pi_{\langle A v \rangle \ds G'} |_{\langle A v \rangle \oplus A L}| = 1$, regarding $\pi_{\langle A v \rangle \ds A L}$ as an operator $\langle v \rangle \oplus L \to \langle A v \rangle \oplus A L$. We now apply Lemma \ref{lem:detSplit} to $\det(A | \langle v \rangle \oplus L)$, estimating
\[
V_{l + 1}(A) \geq \det(A | \langle v \rangle \oplus L) \gtrsim |\pi_{\langle A v \rangle \ds A L}|^{-1} \det(A | \langle v \rangle) \det(A | L)
\gtrsim |A v| \, V_l(A|_G) \, . \qedhere
\]
\end{proof}

\begin{proof}[Proof of Proposition \ref{prop:genSVD}]
Throughout this proof, $k$ is fixed, and $\lesssim, \gtrsim$ mean `$\leq, \geq$ (respectively) up to a multiplicative constant depending only on $k$'.

From the hypothesis $r V_k(A) \leq \det(A | E) \leq V_k(A)$ and Lemma \ref{lem:gelfandProps}, we have the estimate
\[
r \prod_{i = 1}^k c_i(A) \lesssim \prod_{i =1 }^k c_i(A|_E) \lesssim \prod_{i = 1}^k c_i(A) \, ,
\]
and since $c_i(A|_E) \leq c_i(A)$ for each $i \leq k$, it follows that $r c_i(A) \lesssim c_i(A|_E) \leq c_i(A)$. Note that in particular,
\begin{align}\label{eq:ckRealizedbyE}
m(A|_E) : = \min\{|A v| : v \in E, |v| = 1\} = c_k(A|_E) \gtrsim r c_k(A) \, ,
\end{align}
and
\begin{align}\label{eq:gelfandEstimate}
|A|_E| = c_1(A|_E) \gtrsim r c_1(A) \, .
\end{align}

We now give the first step of the induction. Fixing a unit vector $v_1 \in E$ for which $|A v_1| = |A|_E| = c_1(A|_E)$, we shall apply Lemma \ref{lem:oneStep}, obtaining complements $F_1, F_1'$ to $ \langle v_1\rangle,  \langle w_1 \rangle$, where $w_1 := |A v_1|^{-1} A v_1$. In particular, these complements satisfy
\begin{gather*}
\frac{V_k(A)}{c_1(A)} \lesssim V_{k - 1}(A|_{F_1}) \lesssim \frac{V_k(A)}{|A v_1|} \, , \\
|\pi_{\langle v_1 \rangle \ds F_1 }| \lesssim r^{-1}, \quad \text{ and } \quad |\pi_{\langle w_1 \rangle \ds F_1'}| = 1 \, ,
\end{gather*}
where in the second line we have used that $|A v_1| \gtrsim r c_1(A)$ by construction. This ends the first step of the iteration.

At this point, we define $E_2 = E \cap F_1$ (which has dimension $k - 1$). We now derive a lower bound for $\det(A | E_2)$ using Lemma \ref{lem:detSplit}:
\begin{align*}
\det(A | E) = \det( A | \langle v_1 \rangle \oplus E_2) \lesssim |\pi_{\langle v_1 \rangle \ds E_2}| |A v_1|  \det(A | E_2) \lesssim r^{-1} |A v_1| \det(A | E_2) \, .
\end{align*}
On the other hand, for $\det(A| E)$ we have the lower bound
\begin{align*}
\det(A | E) \geq r V_k(A) \gtrsim r |A v_1| V_{k - 1}(A|_{F_1}) \, ,
\end{align*}
and so we conclude that
\begin{align}\label{eq:E2volGrow}
\det(A | E_2) \gtrsim r^2 V_{k - 1}(A|_{F_1}) \, .
\end{align}

We now describe the induction. For notational convenience, we write $F_0 := V, F_0'  := V'$ and $E_1 := E$. The induction hypothesis following the $l $-th step is as follows. We have nested sequences of subspaces (i) $V = F_0 \supset F_1 \supset \cdots \supset F_l$, (ii) $V' = F_0' \supset F_1 \supset \cdots \supset F_l'$ and (iii) $E = E_1 \supset E_2 \supset \cdots \supset E_l$ where, for each $1 \leq i \leq l$, the subspaces $F_i, F_i'$ are $i$-codimensional and the subspace $E_i$ is $(k - (i - 1))$-dimensional. Additionally, for each $1 \leq i \leq l$ we have (iv) a unit vector $v_i \in E_i$ (writing $w_i := |A v_i|^{-1} A v_i$). The objects (i) - (iv) satisfy
\[
E_i := E_{i -1 } \cap F_{i - 1} \, , \quad A F_i \subset F_i' \, , \quad F_{i-1} = \langle v_{i } \rangle \oplus F_{i } \, , \quad \text{ and } \quad F_{i - 1}' = \langle w_i \rangle \oplus F_i' 
\]
and obey the following estimates.
\begin{gather}
\label{eq:volumeGrowControl} \frac{V_{k-(i-1)}(A|_{F_{i-1}})}{c_1(A|_{F_{i - 1}})} \lesssim V_{k - i }(A|_{F_i }) \lesssim \frac{V_{k-(i-1)}(A|_{F_{i-1}})}{|A v_{i }|} \, , \\
\label{eq:projectionsControl} |\pi_{\langle v_i \rangle \ds F_i }| \lesssim r^{-2^{i - 1}}, \quad \text{ and } \quad |\pi_{\langle w_i \rangle \ds F_i'}| = 1 \, , \\
\label{eq:Elvolume} \det( A | E_i) \gtrsim r^{2^{i - 1}} V_{k - (i - 1)} (A|_{F_{i - 1}}) \, .
\end{gather}

Having proved the base case, we now carry out the $(l + 1)$-th step of the induction. By \eqref{eq:volumeGrowControl}, \eqref{eq:projectionsControl}, \eqref{eq:Elvolume} for $i  = l$, it follows from a deduction analogous to that producing \eqref{eq:E2volGrow} that 
\[
 \det(A | E_{l + 1}) \gtrsim r^{ 2^{l}} V_{k - l}  (A|_{F_{l}}) \, ,
\]
where $E_{l + 1} := E_l \cap F_l$ has dimension $k - l$. Similarly to \eqref{eq:gelfandEstimate}, this implies
\[
|A|_{E_{l + 1}}| = c_1(A|_{E_{l  +1}}) \gtrsim r^{2^{l}} c_1(A|_{F_l}) \, .
\]
We now select a unit vector $v_{l + 1} \in E_{l + 1}$ (writing $w_{l + 1} = |A v_{l + 1}|^{-1} A v_{l + 1}$) for which $|A v_{l + 1}| = |A|_{E_l}|$, and apply Lemma \ref{lem:oneStep} with $v = v_{l + 1}$ and $A = A|_{F_l}$, obtaining complements $F_{l + 1}, F_{l + 1}'$ of $\langle v_{l + 1} \rangle, \langle w_{l + 1} \rangle$ for which
\begin{gather*}
A F_{l + 1} \subset F_{l + 1}' \, , \\
\frac{V_{k-l}(A|_{F_{l}})}{c_1(A|_{F_{l }})} \lesssim V_{k - (l + 1) }(A|_{F_{l + 1} }) \lesssim \frac{V_{k-l}(A|_{F_{l}})}{|A v_{l + 1 }|} \, , \\
 |\pi_{\langle v_{l + 1} \rangle \ds F_{l + 1} }| \lesssim r^{-2^{l }}, \quad \text{ and } \quad |\pi_{\langle w_{l + 1} \rangle \ds F_{l + 1}'}| = 1 \, .
\end{gather*}
This completes the induction step, having shown that the induction hypothesis following step $l$ implies that following step $l + 1$. The induction terminates after step $l = k$, since at this point $E_{k + 1} := E_k \cap F_k = \{0\}$.

To complete the proof, we set $F = F_k$ and $F' = F_k'$, and shall check items 1-3. Item 1 is a part of the induction hypothesis following step $l = k$. For Item 2, note that $E = \langle v_1, \cdots, v_k \rangle$, and so
\[
\pi_{F \ds E} = \pi_{F_k \ds \langle v_k \rangle } \circ \cdots \circ \pi_{F_2 \ds \langle v_2 \rangle}  \circ \pi_{F_1 \ds \langle v_1 \rangle} \, .
\]
An analogous equation holds for $\pi_{F' \ds E'}$, using that $E' = A E = \langle w_1, \cdots, w_k \rangle$. Item 2 now follows from  \eqref{eq:projectionsControl} for a suitable choice of the constant $D$ (in fact, we may take $D(k,r) = C_k / r^{2^k - 1}$ where $C_k$ depends only on $k$).

For Item 3, fix $v \in F, |v| = 1$. We assume without loss that $v \notin \ker A$. Using Lemma \ref{lem:detSplit}, we estimate
\begin{align*}
\det(A | E \oplus \langle v \rangle) \gtrsim |\pi_{\langle A v \rangle \ds E'}|^{-1} \det(A | E) \cdot |A v| \gtrsim r V_k(A) \cdot |A v| \, ,
\end{align*}
where in the second lower bound we have appealed to Item 2 (indeed, $|\pi_{\langle A v \rangle \ds E'}| \leq |\pi_{F' \ds E'}|$, and the latter is $\leq 2^k$). On the other hand, $\det(A | E \oplus \langle v \rangle) \leq V_{k + 1}(A)$, and so by Lemma \ref{lem:gelfandProps}, we obtain
\[
|A v| \lesssim r^{-1} \frac{V_{k + 1}(A)}{V_k(A)} \lesssim r^{-1} c_{k + 1}(A) \, .
\]
Upon increasing $D(k,r)$ to accomodate this last estimate, the proof is now complete.
\end{proof}

\section{Finding the dominated splitting}\label{sec:hardbit}

In this section we give the proof of Theorem \ref{thm:main}. Our first objective, comprising much of the section, is to prove the `difficult' implication in Theorem \ref{thm:main}:
\begin{prop}\label{prop:hardDirection}
Let $T$ be a homeomorphism of a compact topological space $X$, and let $A_x^n$ denote a cocycle of injective bounded linear operators acting on a Banach space $V$. 

Assume that
\begin{align}\label{eq:restateMagic}
\max\{c_{k+1}(A_x^n),c_{k+1}(A_{Tx}^n))\} < K\tau^n c_{k}(A_x^{n + 1}) 
\end{align}
for any $x \in X, n \geq 1$. Then, there exists a continuous equivariant splitting $V = \Ec(x) \oplus \Fc(x)$ and a constant $\tilde K > 1$ such that for any $x \in X, n \geq 1$, we have
\begin{align}\label{eq:realizeCK}
m(A^n_x |_{\Ec(x)}) \geq \tilde K^{-1} c_k(A^n_x) \, , \quad \text{ and } \quad |A^n_x|_{\Fc(x)}| \leq \tilde K c_{k + 1}(A^n_x) \, .
\end{align}
\end{prop}

Before sketching the proof let us fix some notation. For the remainder of the section we fix for each $x \in X, n \geq 1$ a pair of $k$-dimensional/codimensional splittings $V = E_n(x) \oplus F_n(x),  E_n'(x) \oplus F_n'(x)$ for which
\begin{gather*}
A^n_x E_n(x) = E_n'(x) \, , \quad A_x^n F_n(x) \subset F_n'(x) \, , \\
\det(A^n_x | E_n(x)) \geq \frac12 V_k(A^n_x) \, , \\
m(A^n_x|_{E_n(x)}) \geq C_k^{-1} c_k(A^n_x) \, , \quad |A^n_x|_{F_n(x)}| \leq C_k c_{k + 1}(A^n_x) \, , \\
\text{and }\quad  |P_n(x)|, |P_n'(x)| \leq C_k \, ,
\end{gather*}
the existence of which is guaranteed by applying Corollary \ref{cor:SVD} to the operator $A_x^n$.
Above, we have written $P_n(x) = \pi_{E_n(x) \ds F_n(x)}, P_n'(x) = \pi_{E_n'(x) \ds F_n'(x)}$. We will also, at times, write $Q_n(x) = \Id - P_n(x), Q_n'(x) = \Id - P_n'(x)$. We note the identities $A^n_xP_n(x)=P_n'(x)A_x^n$, $A^n_xQ_n(x)=Q_n'(x)A_x^n$ which will be used frequently without further comment.

Our strategy will be as follows:
\begin{itemize}
\item We begin by realizing the upper subspace $\Ec(x)$ as the limit of the subspaces $E_n'(T^{-n} x)$ in \S \ref{subsec:findUpperSubspace}. Our method here follows that of \cite{GoQu14b}.
\item In \S \ref{subsec:upperSubspaceContinuous}, we prove the continuity of the map $x \mapsto \Ec(x)$. As a corollary, we obtain that $m(A_x|_{\Ec(x)}) > c_{\Ec}$ for any $x \in X$, where $c_{\Ec} > 0$ is a constant. This fact resolves an important technical issue, as $A_x$, although injective, may have minimal norm zero.
\item In \S \ref{subsec:positiveVolume}, we show that $A^n_x$ realizes a positive proportion of its maximal $k$-dimensional volume growth on the distribution $\Ec(x)$.
\item Finally, in \S \ref{subsec:lowerSubspace}, we apply our `generalized' SVD formulation, Proposition \ref{prop:genSVD}, to the $k$-dimensional subspaces $\Ec(x)$, realizing $\Fc(x)$ as a limit of uniformly `nicely complemented' subspaces to $\Ec(x)$.
\end{itemize}

Having completed the proof of Proposition \ref{prop:hardDirection}, in \S \ref{subsec:converse} we formulate and prove the converse direction, i.e., that the existence of a $k$-dimensional dominated splitting implies \eqref{eq:magic}.

\subsection{Finding the upper subspace $\mathcal E$}\label{subsec:findUpperSubspace}

\begin{prop}\label{prop:findUpperSubspace}
The limit
\[
\mathcal E(x) = \lim_{n \to \infty} E_n'(T^{-n} x) 
\]
exists and is a uniform limit with respect to $x \in X$. Furthermore $\mathcal E$ is an equivariant distribution, i.e., $A_x \mathcal E(x) \subset \mathcal E(T x)$ for every $x \in X$.
\end{prop}


\begin{proof}
We first prove that $\{E_n'(T^{-n} x)\}$ is uniformly Cauchy by giving a sufficiently strong estimate on 
\[ (*) = |Q_n'(T^{-n} x)|_{E_{n + 1}'(T^{- (n + 1)} x)}| \, ,\] 
which implies the bound $d_H(E_n'(T^{-n} x), E_{n + 1}'(T^{- (n + 1)} x) \leq 4 k (*)$ by Remark \ref{rmk:gapEst}.
 
Fix $v' \in E_{n+1}'(T^{-(n + 1)} x)$ with $|v'| = 1$ and let $v \in E_{n+1}(T^{- (n + 1)} x)$ be such that $A_{T^{-(n + 1)} x}^{n + 1} v = v'$. Then,
\begin{align*}
|Q_n'(T^{-n} x) v'| = |A_{T^{-n} x}^n \circ Q_n(T^{-n} x) \circ A_{T^{- (n + 1)} x} v| \leq K_1 C_k  c_{k + 1}(A^n_{T^{-n} x}) |v| \, ,
\end{align*}
where $K_1 = \sup_x |A_x|$. On the other hand,
\begin{align*}
C_k^{-1} c_k(A^{n + 1}_{T^{- (n + 1)} x}) |v| \leq |A^{n + 1}_{T^{- (n + 1)}x} v| = |v'| = 1 \, ,
\end{align*}
and since $v' \in E_{n + 1}'(T^{-(n + 1)}x)$ was arbitrary, we conclude that
\begin{align}\label{eq:largeNgapEst}
(*) \leq K_1 C_k^2 \frac{c_{k + 1}(A^n_{T^{-n} x})}{c_k(A^{n + 1}_{T^{- (n + 1)} x})} \leq K K_1 C_k^2 \tau^n
\end{align}
uniformly with respect to $x \in X$. Because this bound is geometric, we conclude that $\{E_n'(T^{-n} x)\}$ is Cauchy, and so by the completeness of $\Gc_k(V)$ converges uniformly in $x$ to a subspace $\mathcal E(x) \in \Gc_k(V)$.

It remains to check that $A_x \mathcal E(x) \subset \mathcal E(T x)$. Indeed, we shall prove that for any $v \in \mathcal E(x)$, $|v| =1$, we have
\[
d(A_x v, E_{n + 1}'(T^{-n} x)) \leq (4 k + 1) K K_1^2 C_k^2 \tau^n \, ,
\]
hence $A_x v \in \mathcal E(T x)$ since $E_{n + 1}'(T^{-n} x)) \to \mathcal E(T x)$ in $d_H$. To do this, for each $n$ we fix a unit vector $v_n' \in E_n'(T^{-n} x)$ for which $|v - v_n'| \leq 4 k K K_1 C_k^2 \tau^n$, and as before we write $A^n_{T^{-n} x} v_n = v_n'$ for $v_n \in E_n(T^{-n} x)$. Now,
\[
|Q_{n + 1}'(T^{-n} x) \circ A_x v_n'| = |A_{T^{-n} x}^{n + 1} \circ Q_{n + 1}(T^{- n} x) v_n| \leq C_k c_{k + 1}(A^{n +1}_{T^{-n} x}) |v_n| \, ,
\]
whereas $1 = |v_n'| \geq C_k^{-1} c_k(A^n_{T^{-n} x}) |v_n|$, and so
\[
d(A_x v_n', E_{n + 1}'(T^{-n} x)) \leq |Q_{n +1}'(T^{-n} x) \circ A_x v_n'| \leq K K_1^2 C_k^2 \tau^n \, . \qedhere
\]
\end{proof}

\subsection{Continuity of $\mathcal E$}\label{subsec:upperSubspaceContinuous}

\begin{lem}\label{lem:contUpperSubspace}
The function $x \mapsto \mathcal E(x)$ is continuous.
\end{lem}
\begin{proof}
Let $\varepsilon>0$ and choose $n$ large enough that $\frac{1}{4 k}  C_k^2 (K_1 K + 1) \tau^n <\varepsilon/3$ and such that $d(\Ec (z),E_n'(z))<\varepsilon/3$ and $d(\Ec(z),E_{n+1}'(z))<\varepsilon/3$ for all $z \in X$. Let $U_x$ be the set of all $y \in X$ such that
\[ | A_{T^{-(n+1)}y}^{n+1}- A_{T^{-(n+1)}x}^{n+1} | < \tau^n c_{k}(A_{T^{-n}x}^{n+1}) \, ,\]
which is clearly an open neighbourhood of $x$; we claim that $d_H(\Ec(x),\Ec(y))<\varepsilon$ for all $y \in U_x$.

Like before, we appeal to Remark \ref{rmk:gapEst} to estimate
\[
d_H(E_n'(T^{-n}x), E_{n + 1}'(T^{-(n + 1)} y)) \leq 4 k |Q_n'(T^{-n} x)|_{E_{n + 1}'(T^{- (n + 1)} y)}| 
\]
for $y \in U_x$. Fixing such a $y$, let $v' \in E'_{n+1}(T^{-(n+1)}y)$ be a unit vector and choose $v \in E_{n+1}(T^{-(n+1)}y)$ such that $v'=A_{T^{-(n+1)}y}^{n+1} v$, noting that as before, $1 = |v'| \geq C_k^{-1} c_k(A^{n + 1}_{T^{- (n + 1)} x}) |v| $. We have
\begin{align*}
 |Q_n'(T^{-n} x) v'| & =  |Q_n'(T^{-n} x) A_{T^{-(n+1)}y}^{n+1} v| \leq C_k |A_{T^{-(n+1)}y}^{n+1} - A_{T^{-(n+1)}x}^{n+1}| \cdot |v| + |Q_n'(T^{-n} x) \circ A_{T^{-(n+1)}x}^{n+1} v | \\
 & < C_k  \tau^n c_{k}(A_{T^{-n}x}^{n+1}) |v| + |A^n_{T^{-n} x} \circ Q_n(T^{-n} x) \circ A_{T^{-(n + 1)} x} v| \\
 & \leq \big( C_k  \tau^n c_{k}(A_{T^{-n}x}^{n+1}) + K_1 C_k c_{k + 1}(A^n_{T^{-n} x}) \big)  |v| \\
 & < C_k^2 (K_1 K  + 1) \tau^n \, .
\end{align*}
Taking the supremum over unit vectors $v' \in E_{n+1}'(T^{-(n+1)}y)$ we have obtained the bound
\[d_H(E_n'(T^{-n}x), E_{n+1}'(T^{-(n+1)}y)) \leq 4k |Q_n'(T^{-n} x)|_{E_{n + 1}'(T^{- (n + 1)} y)}| <\frac{\varepsilon}{3},\]
and we conclude that for all $y \in U_x$
\begin{align*}
d_H(\Ec(x), \Ec(y)) & \leq d_H(\Ec(x), E_n'(T^{-n}x)) + d_H(E_n'(T^{-n}x), E_{n+1}'(T^{-(n+1)}y)) \\
& + d_H(E_{n+1}'(T^{-(n+1)}y), \Ec(y)) < \varepsilon \, . \quad \quad \quad  \qedhere
\end{align*}
as required.
\end{proof}

The following is an immediate consequence of the continuity of $x \mapsto \Ec(x)$ and the injectivity of $A_x$.
\begin{cor}\label{cor:lowerEcContraction}
There is a constant $c_{\Ec} > 0$ for which
\[
m(A_x|_{\Ec(x)}) > c_{\Ec}
\]
for all $x \in X$.
\end{cor}
\begin{proof}
As before, we write $K_1 =\sup \{|A_x|\colon x \in X\}>0$. For $x \in X$, let $\delta_x:=m(A_x|_{\Ec(x)})$, which is nonzero since $A_x$ is injective. For each $x \in X$ the set
\[U_x:=\left\{y \in X \colon K_1 d_H(\Ec(x),\Ec(y))+|A_x-A_y|<\frac{\delta_x}{2}\right\}\]
is clearly an open neighbourhood of $x$. If $y \in U_x$ and $u \in \Ec(y)$ is a unit vector, choose a unit vector $v \in \Ec(x)$ such that $|u-v|\leq d_H(\Ec(x),\Ec(y))$ and note that
\[|A_yu| \geq |A_xv|-|(A_x-A_y)v|-|A_y|\cdot|v-u|>m(A_x|_{\Ec(x)})-\frac{\delta_x}{2}=\frac{\delta_x}{2}\]
so that $\inf\{m(A_y|_{\Ec(y)})\colon y \in U_x\}>0$. Passing to a finite subcover of the open cover $\{U_x \colon x \in X\}$ completes the proof.
\end{proof}

\subsection{Volume Growth on $\Ec(x)$}\label{subsec:positiveVolume}

\begin{prop}\label{prop:distEst}
There is a constant $K_2 > 0$ such that for any $x \in X$, $n \in \N$, we have
\[
 \log \frac{\det (A^n_{ x}  | \Ec ( x)) }{\det (A^n_{ x} | E_n( x)) } \geq - K_2 \, .
\]
Therefore,
\[
\det (A^n_{ x} | \Ec( x)) \geq r_{\Ec} V_k(A^n_{ x}) \, ,
\]
where $r_\Ec := \frac12 e^{-K_2}$.
\end{prop}
\begin{proof}

Using the multiplicativity of the determinant and that $A_z \Ec(z) = \Ec(Tz)$, we decompose
\[
 \frac{\det (A^n_{ x} | \Ec ( x)) }{\det (A^n_{x} | E_n( x)) }  = \prod_{q = 0}^{n-1} 
 \frac{\det (A_{T^{q} x} | \Ec (T^{q } x)) }{\det (A_{T^{q} x} | A^q_{ x} E_n( x)) } \, .
\]
As usual for such estimates, we will prove a geometric upper bound on 
\begin{align}\label{eq:est1}
d_H(\Ec(T^{ q} x), A_{ x}^q E_n( x)) 
\end{align}
for large $q$.

Our infinite-dimensional setting creates an additional complication, that $A_{T^{q } x}$ may a priori contract on 
$\Ec (T^{q } x), A_{ x}^q E_n( x)$. This is a potential issue for applying the Lipschitz estimate
\eqref{regularity} for the determinant. Fortunately, the first of these is taken care of by Corollary \ref{cor:lowerEcContraction}, and the second will follow from a tight estimate on \eqref{eq:est1}.

We begin by estimating \eqref{eq:est1}. By the triangle inequality,
\begin{align}\label{eq:subspaceEst} \begin{split} 
d_H(\Ec(T^{q } x), A_{ x}^q E_n( x))) \leq &  d_H(\Ec(T^{q } x), A_{ x}^q E_q(x))  + d_H(A_{ x}^q E_q( x), A_{ x}^q E_n( x)) \, .
\end{split} \end{align}
The first term is bounded by $ \frac{4 k}{1 - \tau} K K_1 C_k^2 \tau^q$ by \eqref{eq:largeNgapEst} because $A^q_x E_q(x) = E_q'(x)$. The second term requires an intermediary claim:
\begin{cla}\label{cla:cheatComp}
There exists $Q_0 \in \N$, independent of $x \in X$, such that if $n > q \geq Q_0$, then $ E_n(x)$ complements $F_q( x)$, and
\[
|\pi_{E_n( x) \ds F_q( x)} | \leq C_k + 3 \, .
\]
\end{cla}
\begin{proof}
It suffices to show that $\{F_n(x)\}_{n \in \N}$ is a \emph{uniformly} Cauchy sequence in $\Gc(V)$, i.e., with control independent of $x \in X$. From this and the bound $|\pi_{F_n(x) \ds E_n(x)}| \leq C_k + 1$, we carry over to a bound on $|\pi_{F_q(x) \ds E_n(x)}|$ for $n, q$ sufficiently large using Lemma \ref{lem:openCond} applied to $E = F_n(x), E' = F_q(x), F = E_n(x)$:
\[
|\pi_{F_q(x) \ds E_n(x)}| \leq \frac{|\pi_{F_n(x) \ds E_n(x)}|}{1 - |\pi_{F_n(x) \ds E_n(x)}| d_H(F_n(x), F_q(x))} 
\leq \frac{C_k + 1}{1 - (C_k + 1) d_H(F_n(x), F_q(x))} \, ,
\]
which is $\leq C_k + 2$ when $n, q$ are sufficiently large, hence $|\pi_{E_n(x) \ds F_q(x)}| \leq C_k + 3$.

It remains to show that $\{F_n(x)\}_n$ is uniformly Cauchy. For brevity, we write $P_n(x) = \Id - Q_n(x)$ and $P_n'(x) = \Id - Q_n'(x)$, with $Q_n, Q_n'$ as in the beginning of \S \ref{sec:hardbit}.

For $v \in F_n(x)$, $|v| = 1$, we estimate $P_{n + 1}(x) v$:
\begin{align*}
 C_k^{-1} c_k(A^{n + 1}_x)  |P_{n + 1}(x) v| &  \leq |A^{n + 1}_x \circ P_{n + 1} (x) v| = |P_{n + 1}'(x) \circ A_{T^{ n} x} \circ A_x^n v| \\
 & \leq |P_{n + 1}'(x)| \cdot |A_{T^n x}| \cdot |A^n_x v| \leq C_k^2 K_1 c_{k +1}(A^n_x) \, .
\end{align*}
Thus, we have deduced (again using Remark \ref{rmk:gapEst})
\[
d_H(F_n(x), F_{n + 1}(x)) \leq 4 k C_k^3 K_1 K \tau^n \, .
\]
As before, the fact that this bound is geometric implies the Cauchyness of $\{F_n(x)\}_n$, completing the proof.
\end{proof}

We now complete the estimate of the second term of \eqref{eq:subspaceEst}. Using again Remark \ref{rmk:gapEst}, we shall bound
\begin{align}\label{eq:secondTerm}
d_H(A_{ x}^q E_q( x), A_{x}^q E_n( x)) \leq 4 k |Q_q'( x)|_{A^q_{ x} E_n( x)}| \, .
\end{align}
Fix a unit vector $v' \in A^q_{ x} E_n( x)$, letting $v' = A^q_{ x} v$ for $v \in E_n( x)$. We have the estimate
\[
|Q_q'( x) v'| = |A^q_{ x} Q_q( x) v| \leq C_k^2 c_{k + 1}(A^q_{ x}) |v| \, .
\]
On the other hand, by Claim \ref{cla:cheatComp}, for $n > q \geq Q_0$ we have that
\begin{align}\label{eq:angleLowerEst}
\sin \theta (E_n( x), F_q( x)) \geq (C_k + 3)^{-1} \, ,
\end{align}
and so writing $v = e_q + f_q \in E_q( x) \oplus F_q( x)$, we have $|e_q| \geq (C_k + 3)^{-1} |v|$; using $|\pi_{F_q(x) \ds E_n( x)}| \leq C_k + 2$ we obtain $|f_q| \leq (C_k + 2) |v|$. For $q \geq Q_0$, we now estimate
\[
1 = |v'| = |A^q_{ x} v| \geq |A^q_{ x} e_q| - |A^q_{ x} f_q| \geq \bigg( \frac{1}{C_k(C_k + 3)} c_k(A^q_{ x}) - C_k(C_k + 2) c_{k + 1}(A^q_{ x}) \bigg) |v| \, .
\]
So, there exists $Q_1 \geq Q_0$ sufficiently large so that for $n > q \geq Q_1$,
\[
|v| \leq \frac{2 C_k(C_k + 3)}{c_k(A^q_{ x})} \, .
\]
Plugging back in, we have obtained the bound
\[
\eqref{eq:secondTerm} \leq 4 k \cdot C_k^2 c_{k + 1}(A^q_{ x}) \cdot \frac{2 C_k(C_k + 3)}{c_k(A^q_{ x})} \leq 32 k K K_1 C_k^4 \tau^q \, ,
\]
(recalling that $C_k > 1$) and so in all,
\[
\eqref{eq:subspaceEst} \leq \bigg( \frac{4 k}{1 - \tau} K K_1 C_k^2 + 32 k K K_1 C_k^4 \bigg) \tau^q =: K_2' \tau^q \, .
\]

%
%

We are now ready to apply Lemma \ref{lem:detReg}. Setting $M = \max\{K_1, 2 c_{\Ec}^{-1}\}$, let $L_2, \d_2$ be as in the conclusion of Lemma \ref{lem:detReg} for this choice of $M$ and for $m = k$. By Corollary \ref{cor:lowerEcContraction} and our control on \eqref{eq:subspaceEst}, there exists $Q_2 \geq Q_1$ sufficiently large so that for all $n > q \geq Q_2$, we have
\[
m(A_{T^{q } x}|_{A^q_{ x} E_n( x)}) \geq \frac{c_{\Ec}}{2} 
\]
and $\eqref{eq:subspaceEst} \leq \d_2$. Finally, Lemma \ref{lem:detReg} applies, yielding the estimate
\[
\bigg| \log \frac{\det (A_{T^{q } x} | \Ec (T^{q }x )) }{\det (A_{T^{q } x} | A^q_{ x} E_n( x)) }  \bigg| \leq L_2 d_H(\Ec(T^{q }x), A^q_{ x} E_n( x)) \leq L_2 K_2' \tau^q \, .
\]
Collecting,
\[
\bigg| \log \prod_{q = Q_2}^{n-1} \frac{\det (A_{T^{q} x} | \Ec (T^{q }x )) }{\det (A_{T^{q} x} | A^q_{ x} E_n( x)) } \bigg| \leq L_2 K_2' \sum_{q = Q_2}^{\infty}\tau^q \leq \frac{L_2 K_2' }{1 - \tau} \tau^{Q_2} \, .
\]

It now remains to estimate the initial terms $q < Q_2$. We have
\begin{align*}
\log \prod_{q = 0}^{Q_2 - 1} \frac{\det (A_{T^{q } x} | \Ec (T^{q }x )) }{\det (A_{T^{q } x} | A^q_{ x} E_n( x)) }  =&  \sum_{q = 0}^{Q_2 - 1} \log \det (A_{T^{q } x} | \Ec (T^{q }x )) \\ & - \sum_{q = 0}^{Q_2 - 1} \log \det (A_{T^{q } x} | A^q_{ x} E_n( x))  \, .
\end{align*}
The first term is $\geq  Q_2 k \log c_{\Ec}$, and the second term is $\geq - Q_2 k \log K_1$. This completes the proof.
\end{proof}

It is an immediate consequence of Proposition \ref{prop:distEst} that $A^n_x|_{\Ec(x)}$ realizes a positive proportion of the $k$-th Gelfand number $c_k(A^n_x)$.
\begin{cor}\label{cor:ckRealized}
For any $x \in X, n \geq 1$,
\[
m(A^n_x |_{\Ec(x)}) \geq C^{-1} r_{\Ec}  c_k(A^n_x) \, ,
\]
where $C > 1$ is a constant depending only on $k$, and $r_\Ec$ is as in the statement of Proposition \ref{prop:distEst}.
\end{cor}
\noindent This Corollary follows from Lemma \ref{lem:gelfandProps}; the argument is essentially the same as that of \eqref{eq:ckRealizedbyE} in the beginning of the proof of Proposition \ref{prop:genSVD}.

\subsection{Finding the lower subspace $\Fc$}\label{subsec:lowerSubspace}

We are now in position to find the lower subspace $\Fc(x)$ and deduce its properties. We formulate all that remains of the proof of  Proposition \ref{prop:hardDirection} below as Lemma \ref{lem:lowerSubspace}.

Below, we set $r_{\Ec} = \frac12 e^{- K_2}$, the ratio of volume growth captured by $\Ec(x)$.

As promised in the beginning of Section 3, for each $x \in X, n \geq 1$, we apply Proposition \ref{prop:genSVD} to the operator $A^n_x$ and the subspace $\Ec(x)$, obtaining $k$-codimensional subspaces $\hat F_n(x), \hat F_n'(x)$ complementing $\Ec(x), \Ec(T^n x)$, respectively, for which
\begin{gather*}
A^n_x \hat F_n(x) \subset \hat F_n'(x) \, , \quad |A^n_x|_{\hat F_n(x)}| \leq \hat D c_{k + 1}(A^n_x) \, ,
\quad \text{ and } \quad |\hat P_n(x)|, |\hat P_n'(x)| \leq \hat D \, ,
\end{gather*}
($\hat P_n(x) = \pi_{\Ec(x) \ds F_n(x)}, \hat P_n'(x) = \pi_{\Ec(T^n x) \ds F_n'(x)}$) where $\hat D \geq D(k, r_{\Ec})$ (see Proposition \ref{prop:genSVD}) is chosen so that lastly, we have the estimate
\[
m(A^n_x |_{\Ec(x)}) \geq \hat D^{-1} c_k(A^n_x) 
\]
as in Corollary \ref{cor:ckRealized}.

\begin{lem}\label{lem:lowerSubspace}
For each $x \in X$, the limit
\[
\Fc(x) = \lim_{n \to \infty} \hat F_n(x)
\]
exists and has the following properties for each $x \in X$.
\begin{enumerate}
\item[(i)] $\Fc$ is complemented to $\Ec(x)$ with the bound $|\pi_{\Ec(x) \ds \Fc(x)}| \leq 2 \hat D$.
\item[(ii)] $\Fc$ is an equivariant distribution, i.e., $A_x \Fc(x) \subset \Fc(T x)$, 
\item[(iii)] $|A^n_x |_{\Fc(x)}| \leq C r_{\Ec}^{-1} \hat D c_{k + 1}(A^n_x)$ for all $n \geq 1$, where $C > 1$ depends only on $k$.
\item[(iv)] $x \mapsto \Fc(x)$ is continuous.
\end{enumerate}
\end{lem}

\begin{proof}
As is standard by now, we estimate $d_H(F_n(x), F_{n + 1}(x))$ by bounding $|\hat P_{n + 1}(x) |_{F_n(x)}|$. Fixing $v \in F_n(x), |v | =1 $, we estimate
\begin{align*}
\hat D^{-1} c_k(A^{n + 1}_x)  |\hat P_{n + 1}(x) v| &  \leq |A^{n + 1}_x \circ \hat P_{n + 1} (x) v| = |\hat P_{n + 1}'(x) \circ A_{T^{ n} x} \circ A_x^n v| \\
 & \leq |\hat P_{n + 1}'(x)| \cdot |A_{T^n x}| \cdot |A^n_x v| \leq \hat D^2 K_1 c_{k +1}(A^n_x) \, ,
\end{align*}
concluding that
\[
d_H(\hat F_n(x), \hat F_{n + 1}(x)) \leq 4 k \hat D^3 K_1 K \tau^n 
\]
by Remark \ref{rmk:gapEst}. As before, this geometric bound ensures that $\{\hat F_n(x)\}$ is Cauchy, and so possesses a $k$-codimensional limit $\Fc(x)$.

\medskip

\noindent {\it Proof of Item (i)} We now show that $\Ec(x), \Fc(x)$ are complemented, and control $|\pi_{\Ec(x) \ds \Fc(x)}|$. We obtain these easily from the estimate
\[
d_H(\hat F_n(x), \Fc(x)) \leq \frac{4 k \hat D^3 K_1 K }{1 - \tau} \tau^n 
\]
and Lemma \ref{lem:openCond}, which implies
\[
|\pi_{\Fc(x) \ds \Ec(x)}| \leq \frac{|\pi_{\hat F_n(x) \ds \Ec(x)} | }{1 - |\pi_{\hat F_n(x) \ds \Ec(x)} | d_H(\hat F_n(x), \Fc(x))}
\]
for all $n$ large enough that the right-hand side is well-defined and positive. Taking $n \to \infty$ yields the advertised estimate.

\medskip

\noindent {\it Proof of Item (ii)} To show $A_x \Fc(x) \subset \Fc(T x)$, fix $v \in \Fc(x)$ with $|v| = 1$; similarly to before, we shall estimate $|\hat P_{n + 1}(T x ) A_x v|$. For each $n$, let $v_n \in \hat F_n(x)$ be a unit vector such that $|v - v_n| \lesssim \tau^n$. Then,
\begin{align*}
\hat D^{-1} c_k(A^n_{T x }) |\hat P_{n }(T x) A_x v_{n + 1}| & \leq |A^n_{T x} \hat P_{n}(T x) A_x v_{n + 1}| = |\hat P_n(T x) A^{n + 1}_x v_{n + 1}| \\
& \leq \hat D^2 c_{k + 1}(A^{n + 1}_x) \, ,
\end{align*}
hence
\[
|\hat P_{n }(T x) A_x v_{n + 1}| \leq \hat D^3 \frac{c_{k + 1}(A^{n + 1}_x)}{c_k(A^n_{T x})} \leq \hat D^3 K_1^2 K \tau^n \, .
\]
Like before, this implies $A_x v \in \Fc(T x)$.

\medskip

\noindent {\it Proof of Item (iii)} Fix $v \in \Fc(x), |v| = 1$. We estimate
\begin{align*}
V_{k + 1} (A^n_x) & \geq \det(A^n_x | \Ec(x) \oplus \langle v \rangle ) \\
& \geq C^{-1} |\pi_{\Ec(T^n x) \ds \Fc(T^n x)}|^{-1} \det(A^n_x | \Ec(x)) \cdot |A^n_x v|  & \text{ by Lemma \ref{lem:detSplit},} \\
& \geq C^{-1} \hat D^{-1} r_{\Ec} V_k(A^n_x) |A^n_x v| \, &  \text{ by Proposition \ref{prop:distEst}.}
\end{align*}
The desired estimate now follows from Lemma \ref{lem:gelfandProps}.

\medskip

\noindent{\it Proof of Item (iv)} For brevity, we write $\pi_x = \pi_{\Ec(x) \ds \Fc(x)}$. Fix $x \in X$ and $\varepsilon>0$, and let $n \geq 1$ be large enough that
\[4kK_1\hat{D}^2  \big( 2 + (1 + K) C r_\Ec^{-1} \big) \tau^n<\varepsilon.\]
Define $U_x\subseteq X$ to be the set of all $y \in X$ such that
\[|A^n_x-A^n_y|<\tau^nc_k(A^{n+1}_x)\]
which is clearly an open neighbourhood of $x$; we will show that if $y \in U_x$ then $d_H(\Fc(x),\Fc(y))<\varepsilon$.
 
Let $y \in U_x$, and fix $v \in \Fc(y)$ such that $|v| = 1$. We may estimate
\begin{align*}
\hat D^{-1} c_k(A^{n+1}_x) |\pi_x v| & \leq |A^{n+1}_x \circ \pi_x v|  & \text{by Corollary \ref{cor:ckRealized},} \\
& \leq K_1|A^n_x\circ \pi_x v| =K_1|\pi_{T^n x} \circ A^n_x v|\\
& \leq 2 K_1\hat D |A^n_x - A^n_{y}| + K_1|A^n_y v| \\
& <2K_1\hat D \tau^nc_k(A^{n+1}_x)+ CK_1 r_{\Ec}^{-1} \hat D c_{k + 1}(A^n_{y}) & \text{using Item (iii)}\\
&< K_1\hat{D}(2+(1 + K) Cr_{\Ec}^{-1})\tau^nc_k(A^{n+1}_x),
\end{align*}
where in the last line we have used the inequality
\[c_{k+1}(A^{n}_y)<c_{k+1}(A^{n}_x)+\tau^nc_k\left(A^{n+1}_x\right)<(1 + K) \tau^nc_k\left(A^{n+1}_x\right)\]
which follows from the definition of $U_x$ and the 1-Lipschitz continuity of Gelfand numbers. Taking the supremum over $v \in \Fc(y)$ with $|v|=1$ yields the estimate 
\[
|\pi_x|_{\Fc(y)}| \leq K_1\hat{D}^2  \big( 2 + (1 + K) C r_\Ec^{-1} \big) \tau^n<\frac{\varepsilon}{4k} ,
\]
and so by Remark \ref{rmk:gapEst}
\[d_H(\Fc(x),\Fc(y))<\varepsilon\]
as required.
\end{proof}

\subsection{Converse to Proposition \ref{prop:hardDirection}}\label{subsec:converse}

In this subsection we prove the following `converse' to Proposition \ref{prop:hardDirection}, thereby completing the proof of Theorem \ref{thm:main}.

\begin{lem}
Let $A : X \times \N \to L(V)$ be a continuous cocycle of bounded, injective linear operators on a compact topological space $X$ for which there exists a continuous, $A$-equivariant splitting $V = \Ec (x) \oplus \Fc(x)$ into $k$-(co)dimensional subspaces with the property that
\[\sup_{\substack{u \in \Ec(x), \,v \in \Fc(x)\\ |u|=|v|=1}} \frac{|A_x^n v |}{| A_x^n u| } \leq K\tau^n \, ,\]
where $K > 0$ and $\tau \in (0,1)$ are constants. Then, there exists a constant $ K' > 0$ for which
\[
\max\{c_{k+1}(A_x^n),c_{k+1}(A_{Tx}^n))\} <  K' \tau^n c_{k}(A_x^{n + 1}) 
\]
for any $x \in X, n \in \N$.
\end{lem}

\begin{proof}
First, we observe that from the continuity of $x \mapsto A_x$ in norm, the continuity of $x \mapsto \Ec(x)$ in $\Gc_k(V)$, and the compactness of $X$, we obtain
\[
\sup_{x \in X}|\big(A_x|_{\Ec(x)} \big)^{-1}| = C_0 < \infty \, .
\]
Equipped with this, for $x \in X$ we easily obtain the estimates
\[
|\big(A^{n + 1}_x|_{\Ec(x)} \big)^{-1} | \leq C_0 \min\{ | \big(A^n_x|_{\Ec(x)} \big)^{-1}| , | \big(A^n_{Tx}|_{\Ec(Tx)} \big)^{-1}| \} \, ,
\]
where $m(\cdot)$ is as defined in the beginning of \S \ref{sec:geometry}. We now bound
\begin{align*}
c_{k + 1}(A^n_x) & \leq |A^n_x|_{\Fc(x)}|   \leq K \tau^n m(A^n_x|_{\Ec(x)}) \\
& \leq K C_0 \tau^n m(A^{n + 1}_x|_{\Ec(x)} )  \leq K C_0  \tau^n \knum_k(A^{n + 1}_x) \\
& \leq K C_0 C \tau^n c_k(A^{n + 1}_x) & \text{ by } \eqref{eq:alternativeSnumber} \, .
\end{align*}
Above, the constant $C  > 0$ depends only on $k$, and is the same as appears in Lemma \ref{lem:gelfandProps}. Setting $K' = K C_0 C$ yields the desired upper bound for $c_{k + 1}(A^n_x)$; the analogous estimate for $c_{k + 1}(A^n_{T x})$ proceeds similarly and is omitted.
\end{proof}


\section{Dominated splittings for strongly continuous cocycles}\label{sec:continuous}

In this section we deduce Theorem \ref{thm:continuousTime} from Theorem \ref{thm:main}. Due to a need for additional subscripts, in this section we modify our notation for cocycles $A^n_x$ and $B_x^t$ (in discrete and continuous time respectively), denoting these objects instead by $A(x,n)$ and $B(x,t)$.  Throughout this section we assume that $X$ and $B(x,t)$ satisfy the hypotheses of Theorem \ref{thm:continuousTime}.

We collect some lemmas which will be useful in the proof:
\begin{lem}\label{lem:sot}
There exists a constant $C>0$ such that $|B(x,t)|\leq C$ for all $x\in X$ and $t \in [0,1]$.
\end{lem}
\begin{proof}[Proof of Lemma \ref{lem:sot}]
For each $v \in V$ the compactness of $X \times  [0,1]$ and strong continuity of $(x,t) \mapsto B(x,t)$ implies
\[\sup_{x \in X}\sup_{0 \leq t \leq 1} |B(x,t)v|<\infty.\]
The result follows by the Banach-Steinhaus Theorem.
\end{proof}

\begin{lem}\label{lem:smalltControl}
Suppose that $\Ec \colon X \to \Gc_k(V)$ is continuous. Then there exists $c>0$ such that 
\begin{align}\label{eq:bounded01below}
\inf_{0 \leq t \leq 1} m \big(B(x,t)|_{\Ec(x)} \big) \geq c > 0 \, ,
\end{align}
for every $x \in X$.
\end{lem}

\begin{proof}[Proof of Lemma \ref{lem:smalltControl}]
Since $X \times [0,1]$ is compact it suffices to prove the following: for each $(x,t) \in X \times [0,1]$ there exist an open neighbourhood $U_{x,t}$ of $(x,t)$ and a constant $c_{x,t}>0$ such that for all $(y,s) \in U_{x,t}$ we have $|B(y,s)u|>c_{x,t}$ for every unit vector $u \in \Ec(y)$.

Let $C>0$ be as in Lemma \ref{lem:sot}. Fix $(x,t) \in X \times [0,1]$, let $\delta:=m(B(x,t)|_{\Ec(x)})>0$, and define $c_{x,t}:=\delta/4$. Let $v_1,\ldots,v_n \in V$ be a $\delta/4C$-net for the unit sphere of $\Ec(x)$, and define $U_{x,t}$ to be the set of all $(y,s) \in X \times [0,1]$ such that
\begin{gather*}
\max_{1 \leq r \leq n}\left|B(y,s)v_r-B(x,t)v_r\right|<\frac{\delta}{4} \, , \text{ and}  \\
 d_H(\Ec(x),\Ec(y))<\frac{\delta}{4C} \, .
\end{gather*}
Since $(y,s) \mapsto B(y,s)$ is strongly continuous, each map $(y,s) \mapsto B(y,s)v_r$ is continuous, so $U_{x,t}$ is an open neighbourhood of $(x,t)$.

Let $(y,s) \in U_{x,t}$, let $u \in \Ec(y)$ be a unit vector, and choose a unit vector $v \in \Ec(x)$ such that $|u-v|\leq d_H(\Ec(x),\Ec(y))$. Let $v_r \in \Ec(x)$ be a unit vector such that $|v-v_r|<\delta/4C$, and note that $|u-v_r|<\delta/2C$. We have
\[\big||B(x,t)v_r|-|B(y,s)u|\big|\leq \big||B(x,t)v_r|-|B(y,s)v_r|\big|+\big||B(y,s)v_r|-|B(y,s)u|\big|<\frac{3\delta}{4}\]
and therefore
\[|B(y,s)u|>|B(x,t)v_r|-\frac{3\delta}{4} \geq m(B(x,t)|_{\Ec(x)})-\frac{3\delta}{4}=\frac{\delta}{4}=c_{x,t}>0.\]
The proof is complete. 
\end{proof}
\begin{lem}\label{lem:uniquesplit}
Let $T\colon X \to X$ be a homeomorphism and $A\colon X \times \mathbb{N} \to L(V)$ a continuous linear cocycle of injective operators. Suppose that there exists a pair of continuous equivariant splittings $V=\Ec_1(x)\oplus \Fc_1(x)=\Ec_2(x)\oplus \Fc_2(x)$ of $V$ into closed subspaces such that for all $x \in X$ and $n \geq 1$,
\[|A(x,n)|_{\Fc_i(x)}| \leq Kc_{k+1}(A(x,n)),\]
\[m(|A(x,n)|_{\Ec_i(x)}|) \geq K^{-1}c_{k}(A(x,n)),\]
\[c_{k+1}(A(x,n))\leq K\tau^nc_k(A(x,n))\]
for $i=1,2$, where $K>0$ and $\tau \in (0,1)$ are constants. Then $\Ec_1\equiv \Ec_2$ and $\Fc_1\equiv \Fc_2$.
\end{lem}
\begin{proof}[Proof of Lemma \ref{lem:uniquesplit}]
By symmetry it is sufficient to show that $\Ec_1(x)\subseteq \Ec_2(x)$ and $\Fc_1(x)\subseteq \Fc_2(x)$ for all $x \in X$. We begin by proving the second of these two inclusions. Let $x \in X$ and $v \in \Fc_1(x)$, and write $v=v_1+v_2$ where $v_1 \in \Ec_2(x)$ and $v_2 \in \Fc_2(x)$. For each $n \geq 1$ we have
\begin{align*}K^{-1}c_k(A(x,n))|v_1|&\leq m(A(x,n)|_{\Ec_2(x)})|v_1|\leq |A(x,n)v_1|\\
&=|A(x,n)(v-v_2)|\leq |A(x,n)v|+|A(x,n)v_2|\\
&\leq Kc_{k+1}(A(x,n)) (|v|+|v_2|)\end{align*}
so that dividing both sides by $c_k(A(x,n))$ and letting $n \to \infty$ we obtain $|v_1|=0$; it follows that $v=v_2 \in \Fc_m(x)$ and therefore $\Fc_1(x)\subseteq \Fc_2(x)$ as claimed. 

We next observe that the supremum $M:=\sup_{x \in X}|\pi_{\Fc_2(x)\ds\Ec_2(x)}|$ is finite. Indeed, by the closed graph theorem we have $|\pi_{\Fc_2(x)\ds\Ec_2(x)}|<\infty$ for each $x \in X$, and it is an easy exercise using Lemma \ref{lem:openCond} to show that $|\pi_{\Fc_2(y)\ds\Ec_2(y)}|\leq 2|\pi_{\Fc_2(x)\ds\Ec_2(x)}|$ when $d_H(\Ec_2(x),\Ec_2(y))$ and $d_H(\Fc_2(x),\Fc_2(y))$ are sufficiently small. By the continuity of $\Ec_2$, $\Fc_2$ and compactness of $X$ the finiteness of the quantity $M$ follows.

We may now show that $\Ec_1(x)\subseteq \Ec_2(x)$ for all $x \in X$ as claimed. Let $x \in X$ and $v \in \Ec_1(x)$. For each $n \geq 1$ we may write $v=A(x,n)v_n$ where $v_n \in \Ec_1(T^{-n}x)$, and defining $v_n=w^+_n+w^-_n$ where $w^+_n \in \Ec_2(T^{-n}x)$, $w^-_n\in\Fc_2(T^{-n}x)$ we have $|w^-_n|\leq M|v_n|$. We have $|v|\geq m(A(T^{-n}x,n)|_{\Ec_1(T^{-n}x)})|v_n|\geq K^{-1}c_k(A(x,n))|v_n|$ and therefore
\[
|A(T^{-n}x,n)w^-_n| \leq Kc_{k+1}(A(T^{-n}x,n))|w^-_n|\leq MKc_{k+1}(A(T^{-n}x,n))|v_n| \leq MK^2\tau^n|v|.
\]
It follows that $\lim_{n \to \infty}|v-A(T^{-n}x,n)w^+_n|=\lim_{n \to\infty}|A(T^{-n}x,n)w^-_n|=0$ and therefore
\[v=\lim_{n \to \infty}A(T^{-n}x,n)w^+_n \in \Ec_2(x),\]
 completing the proof of the claim.
\end{proof}

\begin{proof}[Proof of Theorem \ref{thm:continuousTime}] Below, for each $m \geq 1$ we define a cocycle $A_m$ over $T_m:=\phi^{1/m}$ by $A_m(x,n):=B(x,n/m)$. 

\medskip

\noindent {\bf Proof that (a) $\Rightarrow$ (b).} Define $\tau_m := e^{- \gamma / m}$ for every $m \geq 1$. Clearly
\[c_{k+1}(A_m(x,n))=c_k(B(x,n/m)) \leq C\tau_m^n c_{k+1}\big( B(x, \frac{n}{m} + 1)\big),\]
\[c_{k+1}(A_m(T_mx,n))=c_k(B(\phi^{1/m}x,n/m)) \leq C\tau_m^n c_{k+1}\big( B(x, \frac{n}{m} + 1)\big)\]
for all $x \in X$ and $n \geq 0$ in view of \eqref{eq:continuousTime}; on the other hand Lemma \ref{lem:sot} yields
\[
c_k\big( B(x, \frac{n}{m} + 1)\big) \leq \big| B(x, 1-\frac{1}{m}) \big| c_k\big(B(x, \frac{n + 1}{m})\big) \leq C c_k\big(A_m(x, n + 1)\big).
\]
It follows that
\[
\max \{c_{k + 1}(A_m(x, n)), c_{k + 1}(A_m(T_m x, n))\} \leq C \tau_m^n \, c_k(A_m(x,n + 1)) \quad \text{ for all } n \geq 0 \, ,
\]
for every $x \in X$ and $n \geq 1$, so Proposition \ref{prop:hardDirection} applies to the cocycle $A_m$, yielding dominated splittings $V = \Ec_m(x) \oplus \Fc_m(x)$ for each $m \geq 1$. 

We now claim that for every $m \geq 1$ we have $\Ec_m = \Ec_1, \Fc_m = \Fc_1$. Since both $A_1$ and $A_m$ satisfy the hypotheses of Proposition \ref{prop:hardDirection}, we may choose a constant $\tilde K$ such that
\[m(A_1(x,n)|_{\Ec_1(x)}) \geq \tilde K c_k(A_1(x,n)), \quad |A_1(x,n)|_{\Fc_1(x)})| \leq \tilde K c_{k+1}(A_1(x,n)),\]
\[m(A_m(x,mn)|_{\Ec_m(x)}) \geq \tilde K c_k(A_m(x,mn)), \quad |A_m(x,mn)|_{\Fc_m(x)})| \leq \tilde K c_{k+1}(A_m(x,mn))\]
  for every $x \in X$ and $n \geq 1$; but $A_m(x,nm) \equiv A_1(x,n)$ and $T_m^{nm}\equiv T_1^n$, so the hypotheses of Lemma \ref{lem:uniquesplit} are satisfied by the cocycle $A_1$, transformation $T_1$ and splittings $V=\Ec_1(x)\oplus \Fc_1(x)=\Ec_m(x)\oplus \Fc_m(x)$. The claim follows by application of the lemma.
      
Let us now define $\Ec := \Ec_m, \Fc := \Fc_m$. Since $\Ec$ and $\Fc$ are equivariant for every $A_m$ it follows that
\[
B(x,r) \Ec(x) = \Ec(\phi^{r} x)  \quad \text{ and } \quad B(x,r) \Fc(x) \subset \Fc(\phi^{r} x) \quad \text{ for all } r \geq 0, r \in \Q \, .
\]
simply by writing $r=n/m$, $\phi^r=T_m^n$ and $B(x,r)=A_m(x,n)$. We now extend this equivariance result to arbitrary $t \geq 0$. Write $t\geq 0$ as the limit of a sequence of nonnegative rationals $r_n$, and suppose that $u \in \Ec(x)$. For each $n \geq 1$ we have $B(x,r_n)u \in \Ec(\phi^{r_n}x)$, so by the definition of the metric $d_H$ we may choose a sequence of vectors $v_n$ such that $v_n \in \Ec(\phi^tx)$ and $|B(x,r_n)u-v_n|\leq 2|B(x,r_n)u|d_H(\Ec(\phi^{r_n}x),\Ec(\phi^tx))$ for every $n \geq 1$. In particular it follows that $\lim_{n \to \infty} |B(x,r_n)u-v_n| =0$ and therefore
\[B(x,t)u=\lim_{n \to \infty}B(x,r_n)u =\lim_{n \to \infty} v_n \in \Ec(\phi^tx)\]
since $\Ec(\phi^tx)$ is closed. We deduce that $B(x,t)\Ec(x)\subseteq \Ec(\phi^tx)$ as desired; the proof for $\Fc$ is identical.

It remains to check \eqref{eq:continuousDS}. For $t \in [0,1)$, \eqref{eq:continuousDS} is an elementary consequence of Lemma \ref{lem:sot} and Lemma \ref{lem:smalltControl}.  For $t \geq 1$, write $t = n + s$ for $n \in \N, s \in [0,1)$. From \eqref{eq:realizeCK} applied to the cocycle $A_1(x,n) = B(x,n)$, we have that
\begin{align}\label{eq:integerTimeControl}
m(B(x,n)|_{\Ec(x)}) \geq \tilde K^{-1} c_k(B(x,n)) \, , \quad \text{ and } \quad |B(x,n)|_{\Fc(x)}| \leq \tilde K c_{k + 1} (B(x,n)) \, .
\end{align}
Then,
\[
m \big( B(x,t)|_{\Ec(x)} \big) \geq m \big( B(x,n)|_{\Ec(x)} \big) \cdot m \big( B(\phi^n x, s)|_{\Ec(\phi^n x)} \big) \geq c \tilde K^{-1}  c_k(B(x,n))
\]
and 
\[
|B(x,t)|_{\Fc(x)}| \leq |B(x,n)|_{\Fc(x)}| \cdot |B(\phi^n x, s)| \leq C \tilde K c_{k + 1}(B(x,n)) \, ,
\]
which, in light of \eqref{eq:continuousTime}, clearly implies \eqref{eq:continuousDS}.

\medskip

\noindent {\bf Proof that (b) $ \Rightarrow$ (a).} To begin, observe that
\[
m \big( B(x,t) |_{\Ec(x)} \big) \leq C_k c_k(B(x, t)) \qquad \text{ and } \qquad |B(x,t)|_{\Fc(x)}| \geq c_{k + 1}(B(x,t)) \, .
\]
The second estimate is a direct consequence of the definition of the Gelfand numbers $c_k$. The first estimate follows from \eqref{eq:alternativeSnumber} and the definition of $\knum_k(\cdot)$ (see \S \ref{sec:geometry} for details).

Fix $\e \in [0,1]$; the estimate \eqref{eq:continuousDS} implies
\[
c_{k + 1}(B(\phi^\e x,t)) \leq C e^{- \gamma t} m \big( B( \phi^\e x,t)|_{\Ec(\phi^\e x)} \big) \, .
\]
On the other hand, note that $B(x,t + 1) = B(\phi^{t + \e} x, 1 - \e) \circ B(\phi^\e x, t) \circ B(x, \e)$, and so by the equivariance of $\Ec$, 
\[
m \big( B(x, t + 1)|_{\Ec(x)} \big) \geq m \big( B(\phi^{t + \e} x, 1 - \e) |_{\Ec(\phi^{t + \e})} \big) \cdot m \big( B(\phi^\e x, t)|_{\Ec(\phi^\e x) } \big) \cdot m \big( B(x, \e) |_{\Ec(x)} \big) \, .
\]
Using \eqref{eq:bounded01below}, we conclude that
\[
c_{k + 1}(B(\phi^\e x,t)) \leq C c^{-2} e^{- \gamma t} m \big( B(x, t + 1)|_{\Ec(x)} \big) \leq C C_k c^{-2} e^{- \gamma t} c_k(B(x, t + 1)) \, .
\]
As $\e \in [0,1]$ was arbitrary, this completes the proof.
\end{proof}


\section{Quantitative bounds for finite-dimensional cocycles}\label{sec:finiteDimension}

In this section we prove Theorem \ref{thm:finiteDimensional}, a version of Theorem \ref{thm:bogo} which controls the angle of the dominated splitting $\Ec \oplus \Fc$ and vector growth/contraction on $\Ec, \Fc$ respectively. In \S \ref{subsec:hilbertGeometry} we collect some tools in the inner-product space setting, and in \S \ref{subsec:proofFiniteDimension} we show how the methods of \S \ref{sec:hardbit} can be used to prove Theorem \ref{thm:finiteDimensional}.

\subsection{Preliminaries of inner product space geometry}\label{subsec:hilbertGeometry}

We write $(\cdot, \cdot) $ for the Euclidean inner product on $\R^d$ and $\|\cdot\|$ for its norm. For a subspace $E \subset \R^d$, we write $\Pr_E$ for the orthogonal projection onto $E$, and $E^{\perp}$ for the orthogonal complement.  Orthogonal projections permit us to define the Hausdorff distance $d_H$ in this setting by
\[
d_H(E, E') = \|\Pr_E - \Pr_{E'}\| \, ;
\]
this definition of the Hausdorff distance has the nice property that if $E, E'$ have the same dimension, then
\[
d_H(E, E') = \|\Pr_{E^{\perp}}|_{E'}\| = \Gap(E, E') = \|\Pr_{E'^{\perp}}|_E\| = \Gap (E', E) \, ;
\]
see \cite{Ka95} for proof. Here, $\Gap(\cdot, \cdot)$ is as defined in \S \ref{sec:geometry}.

As in previous sections, $\det$ will always refer to the absolute value of the determinant.

\medskip

We require several inner product space analogues of results used in the proof of Theorem \ref{thm:main}. The first is an explicit estimate of the Lipschitz constant of the mapping $B \mapsto \log \det(B)$ as $B$ ranges over the set of invertible $k \times k$ matrices.

\begin{lem}\label{lem:diffMatrices}
Let $B_1, B_2$ be invertible $k \times k$ matrices for which $\|B_1 - B_2\| < \min\{\|B_i^{-1}\|^{-1} : i = 1,2\}=: m$. Then,
\[
\bigg|\log \frac{\det(B_2)}{\det (B_1)}\bigg| \leq k \, \frac{\|B_1 - B_2\|}{m - \|B_1 - B_2\|} \, .
\]
\end{lem}

\begin{proof}
Let $B_t := t B_1 + (1 - t) B_2$, and note that $B_t$ is invertible for all $t \in [0,1]$ by the bound $\|B_1 - B_2\| < \|(B_2)^{-1}\|^{-1}$. A standard computation (using, e.g., the expansion of the determinant into minors and Cramer's rule) implies
\[
\frac{d}{dt} \log \det B_t = \operatorname{Tr}\big( B_t^{-1} (B_2 - B_1) \big) \, .
\]
Applying the Fundamental Theorem of Calculus,
\[
\log \det(B_2) - \log \det (B_1) = \int_0^1 \operatorname{Tr} \big( B_t^{-1} (B_1 - B_2) \big) d t \, ,
\]
To estimate the integrand, note that $B_t^{-1} = (\Id - t B_2^{-1} (B_2 - B_1))^{-1} \circ B_2^{-1}$, so that
\begin{align*}
\| B_t^{-1}(B_1 - B_2)\| &= \big\| \big( (\Id - t B_2^{-1} (B_2 - B_1))\big)^{-1} B_2^{-1} (B_1 - B_2)  \big \| \, , \\
& \leq \sum_{n =0}^\infty t^n \big( \|B_2^{-1}(B_2 - B_1) \| \big)^{n+1}  \\
& \leq \frac{\|B_2^{-1}(B_2 - B_1)\|}{1 - \|B_2^{-1} (B_2 - B_1)\|} \, ,
\end{align*}
having used the von Neumann series to estimate $\big\|\big(\Id - t (B_2^{-1} (B_2 - B_1) \big)^{-1} \big\|$. Using the standard estimate $|\operatorname{Tr}(B)| \leq k \|B\|$ and clearing the denominator, we obtain
\[
 \log \frac{\det(B_2)}{\det (B_1)} \leq k \, \frac{\|B_1 - B_2\|}{\|B_2^{-1}\|^{-1} - \|B_1 - B_2\|} \, .
\]
The desired estimate follows on exchanging the roles of $B_1, B_2$ in the above proof.
\end{proof}

We now give an estimate on the Lipschitz constant of the mapping $E \mapsto \log \det(A|_E)$ defined on the $k$-dimensional Grassmanian $\Gc_k(\R^d)$, where $A \in L(\R^d)$ is assumed invertible. Below, $\kappa(A) := \|A\| \cdot \|A^{-1}\|\geq 1$ is the condition number of $A$.

\begin{lem}\label{lem:lipRegFD}
Let $A \in L(\R^d)$ be invertible and $E_1, E_2  \in \Gc_k(\R^d)$, $k \leq d$. If $d_H(E_1, E_2) \leq \big( 2 \kappa(A) \big)^{-2}$, then
\[
\bigg| \log \frac{\det(A|E_1)}{\det(A | E_2)} \bigg| \leq 36 k \kappa(A)^2 d_H(E_1, E_2) \, .
\]
\end{lem}
\begin{proof}
We apply Lemma \ref{lem:diffMatrices} to $B_1 = A|_{E_1}$, $B_2 = \tilde A := \Pr_{A E_1} \circ A \circ \Pr_{E_2}|_{E_1}$. We first bound $\|(\tilde A)^{-1}\|^{-1}$: it is easy to check from the identities at the beginning of \S \ref{subsec:hilbertGeometry} that 
\[
\| (\Pr_{E}|_E')^{-1}\|^{-1} \geq 1 - d_H(E, E')
\]
for subspaces $E, E' \subset \R^d$ of the same dimension. Applying this to $E = E_2, E' = E_1$ and $E = A E_1, E' = A E_2$ yields
\begin{gather*}
\| (\Pr_{E_2}|_{E_1})^{-1}\|^{-1} \geq 1 - d_H(E_1, E_2) \, , \\
\| (\Pr_{A E_1}|_{A E_2})^{-1} \|^{-1} \geq 1- d_H(A E_1, A E_2) \, ,
\end{gather*}
respectively. Note that
\[
d_H(A E_1, A E_2) \leq \kappa(A) d_H(E_1, E_2) \, ,
\]
and so we conclude
\[
\| (\tilde A)^{-1} \|^{-1} \geq m := (1 - \kappa(A) d_H(E_1, E_2))^2 \|A^{-1} \|^{-1} \, .
\]
We now bound $\|B_1 - B_2\| = \|(A - \tilde A)|_{E_1}\|$:

\begin{align*}
\| (A - \tilde A)|_{E_1} \| &\leq \| A|_{E_1} - A \circ \Pr_{E_2}|_{E_1} \| + \| A \circ \Pr_{E_2}|_{E_1} - \Pr_{A E_1} \circ A \circ \Pr_{E_2}|_{E_1} \|  \\
& \leq \|A\| \cdot \| \Pr_{E_2^\perp}|_{E_1}\| + \|\Pr_{(A E_1)^\perp}|_{A E_2}\| \cdot \|A \circ P_{E_2}\| \\
& \leq \|A\| \big( d_H(E_1, E_2) + d_H(A E_1, A E_2) \big) \leq \|A\| (1 + \kappa(A)) d_H(E_1, E_2) \, .
\end{align*}

To complete the estimate, we decompose
\[
\frac{\det(A | E_1)}{\det ( A | E_2)}  = \frac{\det(A |  E_1) }{\det(\tilde A | E_1)} \cdot \det(\Pr_{A E_1} | A E_2) \cdot \det( \Pr_{E_2} | E_1) \, .
\]
Simplifying the estimate from Lemma \ref{lem:diffMatrices} and using $1 + \kappa(A) \leq 2 \kappa(A)$,
\[
\bigg| \log \frac{\det(A |  E_1) }{\det(\tilde A | E_1)} \bigg| \leq k \frac{\|(A - \tilde A)|_{E_1}\|}{m - \|(A - \tilde A)|_{E_1}\|} \leq k \frac{2 \kappa(A)^2 d_H(E_1, E_2)}{(1 - \kappa(A) d_H(E_1, E_2))^2 - 2 \kappa(A)^2 d_H(E_1, E_2)}\, .
\]
Enforcing $d_H(E_1, E_2) \leq (2 \kappa(A))^{-2}$ now yields an upper bound of $\leq 32 k \kappa(A)^2 d_H(E_1, E_2)$ for this term.

For the remaining terms, we bound
\begin{gather*}
\big| \log \det(\Pr_{E_2} | E_1) \big| \leq \max\{0, k \log \| (\Pr_{E_2}|_{E_1})^{-1}\| \} \leq - k \log (1 - d_H(E_1, E_2)) \, , \\
\big| \log \det(\Pr_{A E_1}| AE_2) \big| \leq \max\{0, k \log \| (\Pr_{AE_1}|_{AE_2})^{-1}\| \} \leq - k \log (1 - \kappa(A) d_H(E_1, E_2)) \, , 
\end{gather*}
and so using the bound $|\log (1 + z)| \leq 2 |z|$ for $z \in [- \frac12, \frac12]$, we get
\[
\big| \log \big( \det(\Pr_{E_2}| E_1) \det(\Pr_{A E_1}| A E_2) \big)  \big| \leq 4 k  \kappa(A) d_H(E_1, E_2) \, .
\]
Combining these yields the advertised estimate.
\end{proof}

\medskip

We complete \S \ref{subsec:hilbertGeometry} with a version of Proposition \ref{prop:genSVD} on $\R^d$. As one might expect, the proof is much simpler (as we have the singular value decomposition to lean on in this setting) and the resulting estimates are far more aesthetically pleasing. 

\begin{prop}\label{prop:genSVDHilbert}
Let $A \in L(\R^d)$ be of rank $\geq k$ and $E \subset \R^d$ a subspace of dimension $k$ for which $\det(A|E) = r V_k(A)$, $r \in (0,1]$.

Define $F = \{v \in \R^d : A v \in (A E)^{\perp} \}$. Then, $F$ is a complement to $E$ in $\R^d$; the splitting $\R^d = E \oplus F$ satisfies
\begin{gather*}
m(A|_E) \geq r \sigma_k(A) \, , \quad \|A|_F\| \leq r^{-1} \sigma_{k + 1}(A) \, , \quad \text{ and }\quad \|\pi_{E \ds F}\| \leq r^{-1} \, .
\end{gather*}
\end{prop}
\begin{proof}
That $F$ so-defined is a complement to $E$ follows from an argument given at the beginning of the proof of Lemma \ref{lem:oneStep}; see in particular \eqref{eq:projectionFormtoshowComplement}. Let $v_1, \cdots, v_k \in E$ be orthonormal vectors corresponding to the singular value decomposition of $A|_E$, written so that $|A v_i| = \sigma_i(A|_E)$ for $1 \leq i \leq k$. 

Define $r_1, \cdots, r_k \in (0,1]$ by $r_i \sigma_i(A) = \sigma_i(A|_E)$; that $\prod_{i = 1}^k r_i = r$ follows from the formulae $\det(A|E) = \prod_{i = 1}^k \sigma_i(A|_E)$ and $V_k(A) = \prod_{i = 1}^k \sigma_i(A)$. In particular, 
\[
m(A|_E) = \sigma_k(A|_E) = r_k \sigma_k(A) \geq r \sigma_k(A) \, .
\]

Next we estimate $\|\pi_{E \ds F}\|$. We note that
\begin{align}\label{eq:projectEstimateHilbert}
\pi_{F \ds E} = \pi_{F \ds \langle v_k \rangle } \circ \pi_{F \oplus \langle v_k \rangle \ds \langle v_{k - 1} \rangle } \circ \cdots \circ \pi_{F \oplus \langle v_2, \cdots, v_k \rangle \ds \langle v_1 \rangle } \, .
\end{align}
To estimate each term, we write $w_i = \|A v_i\|^{-1} A v_i$ and note that
\[
\pi_{\langle v_i \rangle \ds F \oplus \langle v_{i + 1}, \cdots, v_k \rangle } 
= (A|_{\langle v_i \rangle})^{-1} \circ \pi_{  \langle w_i \rangle \ds (A E)^{\perp} \oplus \langle w_{i + 1}, \cdots, w_k \rangle } \circ A|_{\langle v_i, \cdots, v_k \rangle \oplus F} 
\]
for $1 \leq i \leq k-1$; the $i = k$ case is handled similarly. So, for any $1 \leq i \leq k$,
\[
\|\pi_{F \oplus \langle v_{i + 1}, \cdots, v_k \rangle \ds \langle v_i \rangle}  \| = \|\pi_{\langle v_i \rangle \ds F \oplus \langle v_{i + 1}, \cdots, v_k \rangle }  \| \leq \frac{ \sigma_i(A)}{\sigma_i(A|_E)} = r_i^{-1} \, .
\]
Above, we have used that $\|\pi_{W_1 \ds W_2}\| = \|\pi_{W_2 \ds W_1}\|$ for any complements $W_1, W_2$ of $\R^d$; see Problem I.6.31 in \cite{Ka95}. Our desired estimate for $\|\pi_{E \ds F}\| = \|\pi_{F \ds E}\|$  now follows from \eqref{eq:projectEstimateHilbert} and the fact that $\prod_{i = 1}^k r_i = r$. 

The desired estimate for $\|A|_F\|$ now follows by applying Lemma \ref{lem:detSplit} to $\det(A| E \oplus \langle f \rangle)$, where $f \in F, \|f\| = 1$.
\end{proof}


\subsection{Proof of Theorem \ref{thm:finiteDimensional}}\label{subsec:proofFiniteDimension}

Throughout this subsection, $X$ is a general topological space, $T$ a homeomorphism of $X$, and $A : X \times \N \to L(\R^d)$ is a continuous cocycle for some $d \in \N, d > 1$, where $A_x^n$ is invertible for any $(x, n) \in X \times \N$. We take on the hypotheses of Theorem \ref{thm:finiteDimensional}, assuming that for some $1 \leq k < d$, we have
\[
\sigma_{k + 1}(A^n_x) \leq K \tau^n \sigma_k(A^n_x)
\]
for all $x \in X, n \in \N$, and that
\[
\kappa := \bigg( \sup_{x \in X} \|A_x\| \bigg) \cdot \bigg( \sup_{x \in X} \|A_x^{-1}\| \bigg) < \infty \, .
\]

With these assumptions, the arguments of \S\ref{sec:hardbit} go through nearly unchanged; indeed, these arguments are far simpler and rely only on the classical singular value decomposition for linear operators on $\R^d$. 

We will now sketch this argument, along the way recording an estimate for use in the proof of Theorem \ref{thm:finiteDimensional}. Let $n$ be sufficiently large so that $K \tau^n < 1$, hence $\sigma_{k + 1}(A_x^n) < \sigma_k(A^n_x)$. For such $n$, let $E_n(x) \subset \R^d$ denote the $k$-dimensional subspace corresponding to the top $k$ singular values of $A_x^n$; write $E_n'(x) = A_x^n E_n(x)$. Following the procedure outlined in the proof of Proposition \ref{prop:findUpperSubspace}, one obtains the estimate
\begin{align}\label{eq:EconvFD}
d_H( E_n'(T^{-n} x), E_{n+ q}'(T^{- (n + q)} x)) \leq \kappa K \frac{\tau^n}{1 - \tau} \, ,
\end{align}
for any $x \in X$, $q \geq 1$ and $n$ sufficiently large. Thus we construct the upper $k$-dimensional subspace
\[
\Ec(x) := \lim_{n \to \infty} E_n'(T^{-n} x) \, .
\] 
The continuity of $x \mapsto \Ec(x)$ may be deduced the same arguments as in Lemma \ref{lem:contUpperSubspace}; as an alternative, one may show directly that $E_n(x)$ depends continuously on $x \in X$, and therefore  $\Ec$ is a uniform limit of continuous functions, hence continuous.

\medskip

We now come to the analogue of Proposition \ref{prop:distEst}, where it is proven that the upper subspace $\Ec$ realizes a uniformly controlled positive fraction of the maximal $k$-dimensional volume growth of $A_x^n$. This is crucial to obtaining the advertised concrete estimates for the dominated splitting, and so we formulate this analogue precisely.

\begin{lem}\label{lem:posFracVolGrowFD}
The constant
\[
R_{\Ec} := \inf_{\substack{x \in X \\ n \geq 1}} \frac{\det(A^n_x|_{\Ec(x)})}{\prod_{i = 1}^k \sigma_i (A^n_x)}
\]
is nonzero and obeys the lower bound
\begin{align}
\label{eq:rEc}
- \log R_{\Ec} \leq 2 k \log \kappa \cdot \bigg \lceil \frac{\log (12 \kappa^3 K)  -  \log (1 - \tau) }{- \log \tau} \bigg \rceil + 2 k (1 - \tau)^{-1}
\end{align}
\end{lem}
Postponing a proof of the estimate \eqref{eq:rEc} for the moment, let us complete the proof of Theorem \ref{thm:finiteDimensional}. We first obtain the lower subspace $\Fc$ in a way completely analogous to the methods of \S \ref{subsec:lowerSubspace}: for each $x \in X, n \in \N$, we let $\hat F_n(x)$ denote the complement to $\Ec(x)$ guaranteed in the conclusion of Proposition \ref{prop:genSVDHilbert}. For these, we have the estimates
\[
\|\pi_{\Ec(x) \ds \hat F_n(x)} \| \leq R_{\Ec} \, ,  \quad m(A^n_x|_{\Ec(x)}) \geq R_{\Ec} \sigma_k(A^n_x) \quad \text{ and } \|A^n_x|_{\hat F_n(x)} \| \leq R_{\Ec}^{-1} \sigma_{k + 1}(A^n_x) \, ,
\]
all of which follow from the conclusions of Proposition \ref{prop:genSVDHilbert}.

Following the arguments in Lemma \ref{lem:lowerSubspace}, the lower subspace is realized as the limit $\Fc(x) = \lim_{n \to \infty} \hat F_n(x)$; the estimate on $\|\pi_{\Ec(x) \ds \hat F_n (x)}\|$ passes to an identical estimate on $\|\pi_{\Ec(x) \ds \Fc(x)}\|$. It remains, then, to estimate $\|A^n_x|_{\Fc(x)}\|$, which again we do by the exact same method: for $f \in \Fc(x), \|f\| = 1$ we bound
\[
V_{k + 1}(A^n_x) \geq \det(A^n_x| \Ec(x) \oplus f) \geq R_{\Ec} \det(A^n_x | \Ec(x)) \cdot \|A^n_x f\| \geq R_{\Ec}^2 V_k(A^n_x) \|A^n_x f \| \, ,
\]
hence $\|A^n_x|_{\Fc(x)}\| \leq R_{\Ec}^{-2} \sigma_{k + 1}(A^n_x)$, as desired.

\medskip

To finish, we sketch a proof of \eqref{eq:rEc}.

\begin{proof}[Proof of Lemma \ref{lem:posFracVolGrowFD}] 
Let $n$ be large enough so that $K \tau^n < 1$, hence $\sigma_{k + 1}(A^n_x) < \sigma_k(A^n_x)$ and so the corresponding $k$-dimensional subspace $E_n(x) \subset \R^d$ is well-defined. Thus we have
\[
{ \det(A^n_x | \Ec( x)) \over \prod_{i = 1}^k \sigma_i(A^n_x) } = {\det(A^n_x | \Ec( x)) \over \det(A^n_x | E_n( x)) } = 
\prod_{q = 0}^{n-1} { \det(A_{T^{q } x} | \Ec(T^{q } x) )  \over \det(A_{T^{q } x} | A^q_{ x} E_n( x) )  } \, .
\]
The task at hand is to repeat the proof of Proposition \ref{prop:distEst} while keeping track of the estimates made along the way. Like before, the idea is to estimate separately the tail $q \geq Q$ and the initial segment $q < Q$, where $Q$ is chosen sufficiently large so that $A^q_{ x} E_n( x) \approx A^q_{ x} E_q( x) \approx \Ec(T^q x)$ for all $n > q \geq Q$. The value of $Q$ will be specified later on; for the moment we assume that $Q$ is large enough so that $K \tau^Q < 1$ and that $n > q \geq Q$, so that in particular $E_q(x)$ is well-defined.

We start by collecting analogues of the estimates on $d_H(A^q_x E_q(x), \Ec(T^q x))$ and $d_H(A^q_x E_n(x), A^q_x E_q(x))$. The first of these comes from \eqref{eq:EconvFD}. For the second, we repeat the estimates in Claim \ref{cla:cheatComp}, obtaining for $q < n$ that
\[
d_H(E_n^{\perp}(x), E_q^{\perp}(x)) \leq \kappa K \frac{\tau^q}{1 - \tau} \, .
\]
With this estimate, we now estimate $d_H(A^q_x E_n(x), A^q_x E_q(x))$: let $v' \in A^q_x E_n(x)$ be a unit vector, and let $v' = A^q_x v$ where $v \in E_n(x)$. Writing $v = e_q + f_q$ according to the splitting $\R^d = E_q(x) \oplus E_q^\perp(x)$, note that
\[
\|f_q\| \leq \|\Pr_{E_q^\perp(x)}|_{E_n(x)}\| \|v\| \leq \bigg( \kappa K \frac{\tau^q}{1 - \tau} \bigg) \|v\| \, .
\]
Take $Q$ large enough so that the parenthetical term is $\leq 1/\sqrt 2$ for $q \geq Q$. Now we bound $\|e_q\|^2 = \|v\|^2 - \|f_q\|^2 \geq \frac12 \|v\|^2$, and obtain the bound
\[
1 = \|v'\| = \|A^q_x v\| \geq \|A^q_x e_q\| - \|A^q_x f_q\| \geq \frac{1}{\sqrt 2} \sigma_k(A^q_x) (1 - K \tau^q) \|v\| \, .
\]
Taking $Q$ larger so that $1 - K \tau^Q \geq \frac12$,
\[
\| \Pr_{A_x^q E_q^\perp(x)} v' \| = \| A^q_x \circ \Pr_{E_q^\perp(x)} v \| \leq \sigma_{k + 1}(A^q_x) \cdot \|v\| \leq \frac{K \tau^q}{1 - K \tau^q} \leq 2 K \tau^q \, .
\]
At last we collect our estimates, and obtaining
\[
d_H(\Ec(T^{q } x) , A^q_{ x} E_n( x) ) \leq \bigg( \frac{\kappa K}{1 - \tau} + 2 K \bigg)  \tau^q
\]
for $n > q \geq Q$.

The third and last condition we place on $Q$ is to ensure that the RHS above is $\leq (2 \kappa)^2$, so that $E_1 = \Ec(T^q x), E_2 = A^q_x E_n(x)$ obey the hypotheses of Lemma \ref{lem:lipRegFD}. All three conditions are met if $Q$ is chosen as
\begin{align}\label{eq:defineQ}
Q := \bigg\lceil \frac{ \log (3 K \kappa^3) - \log(1 - \tau) }{- \log \tau} \bigg\rceil \, ;
\end{align}
equivalently, $Q$ is the smallest nonnegative integer for which $\tau^Q \leq (1 - \tau) / 3 K\kappa^3$.

For $q \geq Q$, after applying Lemma \ref{lem:lipRegFD} we now have
\[
\bigg| \log { \det(A_{T^{q } x} | \Ec(T^{q } x) )  \over \det(A_{T^{q } x} | A^q_{ x} E_n( x) )  } \bigg| 
\leq
108 k \kappa^3 K \frac{\tau^q}{1 - \tau} \, .
\]
Thus we have the estimate
\[
\sum_{q = Q}^{n - 1} \bigg| \log { \det(A_{T^{q } x} | \Ec(T^{q } x) )  \over \det(A_{T^{q } x} | A^q_{ x} E_n( x) )  } \bigg| \leq 108 k \kappa^3 K \frac{\tau^Q}{ (1 - \tau)^2} \leq \frac{36 k}{1 - \tau} \, .
\]
for the tail of the product. For the remaining terms we use the simple estimate $|\log \det(A | E)| \leq k \log \kappa(A)$, where $\dim E = k$ and $\kappa(A) = \|A\|\|A^{-1}\|$ is the condition number of $A$. Collecting, we have

\begin{gather*}
\sum_{q = 0}^{n-1} \bigg| \log { \det(A_{T^{q } x} | \Ec(T^{q } x) )  \over \det(A_{T^{q } x} | A^q_{ x} E_n( x) )  } \bigg| \leq 2 k Q \log \kappa + \frac{36 k}{ 1 - \tau} \\
= 2 k \log \kappa \cdot \bigg \lceil \frac{\log (3 \kappa^3 K)  -  \log (1 - \tau) }{- \log \tau} \bigg \rceil + \frac{36 k}{ 1 - \tau} \, ,
\end{gather*}
completing the estimate.
\end{proof}

\section{Acknowledgments}

Ian Morris was supported by the Engineering and Physical Sciences Research Council (grant number EP/L026953/1). This research was initiated at the workshop ``Random Dynamical Systems and Multiplicative Ergodic Theorems'' held at the Banff International Research Station; the authors would like to thank the conference organizers for facilitating this interaction.

\bibliography{badosp}
\bibliographystyle{siam}

\end{document}